\documentclass[11pt]{amsart} 
\usepackage{amsmath,amssymb,amsfonts,amsthm,amscd}
\usepackage{graphics,latexsym,multicol,euscript}
\usepackage{graphicx}
\usepackage{flafter}
\usepackage{hyperref}
\usepackage{tikz-cd}
\usepackage{booktabs}

\makeatletter

\def\jobis#1{FF\fi
  \def\preedicate{#1}%
  \edef\preedicate{\expandafter\strip@prefix\meaning\preedicate}%
  \edef\job{\jobname}%
  \ifx\job\preedicate
}

\makeatother

\if\jobis{proposal}%
 \def\try{subsection}%
\else
  \def\try{section}%
\fi

\numberwithin{equation}{section}
 
\theoremstyle{plain}
\newtheorem{theorem}{Theorem}[\try]
\newtheorem{corollary}[theorem]{Corollary}
\newtheorem{lemma}[theorem]{Lemma}

\newtheorem{example}[theorem]{Example}
\newtheorem{proposition}[theorem]{Proposition}
\newtheorem{definition-lemma}[theorem]{Definition-Lemma}

\newtheorem{definition}[theorem]{Definition}
\newtheorem{remark}[theorem]{Remark}

\newtheorem{notation}[theorem]{Notation}

\newcommand{\maxid}{\mathfrak{m}}

\newcommand{\idvar}{\mathfrak{a}}

\newcommand{\rk}{\operatorname{rank}}

\newcommand{\symalg}{\mathfrak{g}_F}
\newcommand{\hiddensection}[1]{} 

\usepackage{stmaryrd}

\newcommand{\im}{\operatorname{im}}

\usetikzlibrary{shapes.geometric, arrows}

\tikzstyle{startstop} = [rectangle, rounded corners, minimum width=3cm, minimum height=1cm,text centered, draw=black, fill=white]
\tikzstyle{process} = [rectangle, minimum width=3cm, minimum height=1cm, text centered, draw=black, fill=white]
\tikzstyle{decision} = [diamond, aspect=2, minimum width=3cm, minimum height=1cm, text centered, draw=black, fill=white]
\tikzstyle{arrow} = [thick,->,>=stealth]

\title{A dimension bound for symmetrizer groups of projective hypersurfaces}
\author{JeGyeong Jung} 
\date{\today}
\address{Department of Mathematical Sciences \\  
KAIST\\  
291 Daehak-ro\\
Yuseong-gu\\
Daejeon, 34141, Republic of Korea}
\email{jgjung@kaist.ac.kr}
\begin{document}

\begin{abstract}
Let $X$ be a projective hypersurface that is not a cone. The symmetrizer group of $X$ is an algebraic group that parametrizes hypersurfaces whose Jacobian ideal coincides with that of $X$. We prove that if the locus of points of multiplicity $d-1$ does not contain a line, then the nilpotent part of the Lie algebra of the symmetrizer group has dimension at most $2$, and consequently the symmetrizer group has dimension at most $\dim X+2$. Moreover, we show that if this locus has only finitely many lines, then the nilpotent part of the Lie algebra has dimension at most $4$, yielding the bound $\dim X+3$ for the symmetrizer group. To achieve this, we establish a connection between a class of singularities, called quasi-vertices, on $X$ with highly degenerate tangent cones and the unipotent part of its symmetrizer group.
\end{abstract}

\maketitle
\section{\bf Introduction}
Throughout this paper, we work over the complex numbers $\mathbb{C}$. Fix an integer $d$ and a vector space $V$ of dimension $n+1$. To avoid triviality, we assume that $d$ and $n+1$ are at least $3$. Let $\mathbb{P}V$ be the projective space parametrizing the one-dimensional subspaces of $V$ and $\operatorname{Sym}^d V^*$ be the vector space of symmetric $d$-forms on $V$.

For $F \in \operatorname{Sym}^d V^*$, consider the $\mathbb{C}$-linear map $\partial F : V \to \operatorname{Sym}^{d-1} V^*$, $v\mapsto \partial_vF$, where $\partial_v F\in \operatorname{Sym}^{d-1}V^*$ is the symmetric $(d-1)$-form given by
\[
\partial_v F(v_1, \dots, v_{d-1}) = F(v, v_1, \dots, v_{d-1}).
\]
It is well-known that $\ker \partial F = 0$ if and only if the hypersurface $Z(F) \subset \mathbb{P}V$ defined by the homogeneous polynomial of degree $d$ corresponding to $F$ is not a cone. Let $\operatorname{Sym}^d_o V^*$ be the Zariski open subset of $\operatorname{Sym}^d V^*$ consisting of all $F\in \operatorname{Sym}^dV^*$ such that $Z(F)$ is not a cone, and let $\operatorname{Gr}(n+1, \operatorname{Sym}^{d-1} V^*)$ be the Grassmannian of $n+1$-dimensional subspaces of $\operatorname{Sym}^{d-1} V^*$. Then we have the morphism
\[J : \operatorname{Sym}^d_o V^* \to \operatorname{Gr}(n+1, \operatorname{Sym}^{d-1} V^*), \quad F \mapsto \operatorname{im} \partial F.\]
The fibers of this morphism were first studied by Mammana in \cite{Mammana}, where he completely classified those $F$ for which $J^{-1}(J(F))\ne \mathbb{C}^\times F$ by determining their explicit forms. In \cite[Section 4.b]{CG79}, Carlson and Griffiths proved that for a general $F \in \operatorname{Sym}^d_o V^*$, the fiber $J^{-1}(J(F))$ is precisely $\mathbb{C}^\times \cdot F$, and they used this result in the study of the variation of Hodge structures of hypersurfaces. For related work on this morphism, see also \cite{UY}, \cite{Wang}, \cite{Hwang}, and \cite{canino2025detectingdirectsumstensors}.

In \cite{Hwang}, Hwang provided the following geometric description of $J$.

\begin{theorem}[{\cite[Theorem 1.5]{Hwang}}]\label{t_Hwang_1.5}
    Let $F\in \operatorname{Sym}^d_oV^*$. Then there is a connected abelian algebraic subgroup $G_F\subset \operatorname{GL}(V)$ canonically associated to $J(F)$, which contains $\mathbb{C}^\times\cdot \operatorname{Id}_V$, such that the fiber $J^{-1}(J(F))$ is a principal homogeneous space of the group $G_F$. 
\end{theorem}

The group $G_F$ is called the \emph{symmetrizer group} of $F$, and the Lie algebra $\symalg \subset \operatorname{End}(V)$ associated to $G_F$ is called the \emph{symmetrizer algebra} of $F$. In \cite{Hwang}, Hwang explained the decomposition of $\symalg$ as follows. By Theorem 3.1.1 and 3.4.7 of \cite{springer1998}, $G_F$ decomposes as
\[
G_F / (\mathbb{C}^\times \cdot \operatorname{Id}_V) = G_F^\times \times G_F^+,
\]
where $G_F^\times$ is a diagonalizable group and $G_F^+$ is a vector group. Let $\symalg$ be the Lie algebra corresponding to $G_F$. Then we have the corresponding decomposition
\begin{align*}
    \symalg / (\mathbb{C} \cdot \operatorname{Id}_V) = \symalg^\times \oplus \symalg^+,
\end{align*}
where $\symalg^+$ (resp. $\symalg^\times$) consists of nilpotent (resp. semisimple) elements. The nilpotent part $\symalg^+$ is closely related to the singularities of $Z(F)$.

\begin{proposition}[{\cite[Proposition 3.2(ii)]{Hwang}; see also \cite[page 134]{Mammana} and \cite[Theorem 1.1]{Wang}}]
\label{t_Hwang_1.6}
    Let $F\in \operatorname{Sym}^d_oV^*$ and $g\in \symalg^+$ a nonzero element such that $g^2=0$. Then every point of $\mathbb{P}(\im g)\subset \mathbb{P}V$ is a singular point of $Z(F)$ with multiplicity $d-1$.
\end{proposition}

    In \cite{Hwang}, Hwang first raised the question of how large $\dim_{\mathbb{C}} \symalg^+$ can be when $Z(F)$ has only finitely many singular points of multiplicity $d-1$. We answer this question by proving Theorem \ref{t_main} and Theorem \ref{t_main_2}. In addition, we classify the Artinian local ring \[R_F:=\mathbb{C}\oplus \mathfrak{g}_F^+\] with the maximal ideal $\mathfrak{g}_F^+$.
    
\begin{theorem}\label{t_main}
    Let $F\in \operatorname{Sym}^d_oV^*$. Assume that the locus of points of $Z(F)\subset\mathbb{P}V$ with multiplicity $d-1$ does not contain a line. (This is the case, for example, if $Z(F)$ has only isolated singularities.)
            \begin{enumerate}
                \item The ring $R_F$ is isomorphic to 
                \[\mathbb{C},\quad \mathbb{C}[x]/(x^2),\quad \text{or}\quad \mathbb{C}[x]/(x^3),\]
                and, in particular, the sharp bound $\dim_\mathbb{C}\mathfrak{g}_F^+\le 2$ holds.
                \item We have the sharp bound $\dim G_F\le \dim_\mathbb{C}V$.
            \end{enumerate}
\end{theorem}

\begin{theorem}\label{t_main_2}
    Let $F\in \operatorname{Sym}^d_oV^*$. Assume that the locus of points of $Z(F)\subset\mathbb{P}V$ with multiplicity $d-1$ has only finitely many lines of $\mathbb{P}V$. (This is the case, for example, if the singular locus of $Z(F)$ has only finitely many lines.)
    \begin{enumerate}
        \item We have the sharp bound $\dim_\mathbb{C}\mathfrak{g}_F^+\le 4$, and $R_F$ is isomorphic to one of the following $10$ rings:
        \begin{center}
        \begin{tabular}{c|c}
            $\dim_\mathbb{C}\maxid$& $R_F$\\
            \hline
            $0$& $\mathbb{C}$\\
            $1$& $\mathbb{C}[x]/(x^2)$\\
            $2$& $\mathbb{C}[x]/(x^3)$,\quad $\mathbb{C}[x,y]/(x,y)^2$\\
            $3$& $\mathbb{C}[x]/(x^4)$,\quad $\mathbb{C}[x,y]/(xy,x^3,y^2)$,\quad $\mathbb{C}[x,y]/(x^2,y^2,(x,y)^3)$\\
            $4$& $\mathbb{C}[x]/(x^5)$,\quad $\mathbb{C}[x,y]/(xy,x^3,y^3)$,\quad $\mathbb{C}[x,y]/(xy,y^2-x^3,x^4)$.
        \end{tabular}
        \end{center}
        \item Moreover, $\dim G_F\le \dim_\mathbb{C}V+1$.
    \end{enumerate}
\end{theorem}

To obtain Theorem \ref{t_main}, we analyze a class of singularities with highly degenerate tangent cones.

\begin{definition}
Let $X\subset \mathbb{P}^n$ be a projective hypersurface of degree $d$. 
A point $x\in X$ is called a \textbf{quasi-vertex} if the tangent cone to $X$ at $x$ is the non-reduced hypersurface in $\mathbb{C}^n$ cut out by $L^{d-1}$ for some homogeneous linear polynomial $L$.
\end{definition}

 For example, the $E_6$ singularity on a cubic surface is a quasi-vertex. In Section \ref{section_4}, we will see that a quasi-vertex behaves like a vertex of a cone. The locus of quasi-vertices in $Z(F)$ can be described in terms of $\symalg^+$:

\begin{theorem} \label{t_quasi-vertices} Let $F\in\operatorname{Sym}^d_oV^*$.
\begin{enumerate}
    \item There is a one-to-one correspondence
    \begin{align*}\{g\in \mathbb{P} \symalg^+\mid \rk g=1\}&\longleftrightarrow \{\text{quasi-vertices of $Z(F)$}\}\\
    g&\longmapsto \mathbb{P} (\im g).
    \end{align*}
    \item Assume that the locus of points of $Z(F)$ with multiplicity $d-1$ does not contain a line. The following are equivalent.\begin{enumerate}
        \item $\symalg^+\ne 0$.
        \item $Z(F)$ has a quasi-vertex and it is the unique point of $Z(F)$ with multiplicity $d-1$.
        \item $Z(F)$ has a quasi-vertex.
    \end{enumerate}
\end{enumerate}
\end{theorem}

The proof of Theorem \ref{t_main_2} requires more than the analysis of quasi-vertices. A key ingredient in is a detailed study of the irreducible components of the square-zero locus
\[N_F^2:=\{g\in \mathbb{P}\symalg^+\mid g^2=0\},\]
which is carried out in Theorem \ref{t_component}.
\medskip

We now provide a brief overview of the paper. 
In Section \ref{section_2}, we recall several definitions and results, present some generalizations, and establish several foundational results that will be used throughout the paper.
In Section \ref{section_3}, we establish basic properties of the Hessian of a symmetric $d$-form.
In Section \ref{section_4}, we investigate the behavior of quasi-vertices and prove Theorem \ref{t_quasi-vertices}.
In Section \ref{section_5}, we first establish the key ingredients for the dimension bounds including Theorem \ref{t_component} and then prove Theorems \ref{t_main} and \ref{t_main_2}.
Finally, in Section \ref{section_6}, we present several corollaries. In particular, using the classifications of Bruce–Wall \cite{BW} and Viktorova \cite{Viktorova}, we present a complete list of isolated singularities of cubic surfaces and cubic threefolds whose symmetrizer algebra has a nontrivial nilpotent part. \\ \\
\textbf{Acknowledgements.} The author would like to thank his advisor, Professor Yongnam~Lee, for suggesting this topic and for his guidance and valuable comments. The author is indebted to the referee for a careful reading of the manuscript and for many detailed suggestions, which improved both the exposition and the content of the paper. The author is also grateful to Professor Jun-Muk Hwang for helpful discussions, to Professor Francesco Russo for kindly providing access to Carmelo Mammana’s paper~\cite{Mammana}, and to Haesong Seo, Chiwon Yoon and Hyukmoon Choi for their detailed comments on an earlier draft of this paper. The author is supported by the Institute for Basic Science (IBS-R032-D1).

\section{\bf Preliminaries}\label{section_2}
\begin{notation}\label{n} We provide a list of the notation used in this paper.\begin{enumerate}
    \item Let $V^*$ be the dual space of $V$. We denote by $\operatorname{Sym}^dV^*$ the $\mathbb{C}$-vector space of all symmetric $d$-linear forms on $V$. That is, a $d$-linear form $F:V\times\cdots\times V\to \mathbb{C}$ is contained in $\operatorname{Sym}^dV^*$ if and only if for any $v_1,...,v_d\in V$ and any permutation $\sigma$ of the permutation group on $\{1,...,d\}$, $F(v_1,...,v_d)=F(v_{\sigma(1)},...,v_{\sigma(d)})$.
    \item For a nonzero vector $v\in V$, $[v]$ denotes the image of $v$ in $\mathbb{P} V$.
    \item Let $F\in\operatorname{Sym}^dV^*$. For a point $[\xi]\in Z(F)$, $\operatorname{mult}_{[\xi]}(F)$ denotes the multiplicity of $Z(F)$ at $[\xi]$. Note that $\operatorname{mult}_{[\xi]}(F)\ge m$ if and only if 
    \[F(\underbrace{\xi,...,\xi}_{d-m+1},v_1,...,v_{m-1})=0\] 
    for any $v_1,...,v_{m-1}\in V$. We denote by $\operatorname{Sym}^d_oV^*$ the Zariski open set of $\operatorname{Sym}^dV^*$ consisting of all $F$ such that $Z(F)$ is not a cone.
    \item For $F\in \operatorname{Sym}^dV^*$, $Z(F)$ is said to be a \textbf{cone}, if it is a cone over a hypersurface in $\mathbb{P}^{n-1}\subset \mathbb{P}V$. Note that $Z(F)$ is a cone with vertex $[\xi]\in Z(F)$ if and only if $\operatorname{mult}_{[\xi]}(F)=d$.
    \item Let $S\subset V$ be a subset. The subspace of $V$ generated by $S$ is denoted by $\langle S\rangle$.
    \item For $g\in \operatorname{End}V$, $\rk g$ denotes the dimension of $\im g$.
    \item The vector space of homogeneous polynomials of degree $d$ in $\mathbb{C}[x_0,...,x_n]$ is denoted by $\mathbb{C}[x_0,...,x_n]_d$.
\end{enumerate}
\end{notation}

\begin{definition}[{\cite[Definition 2.1]{Hwang}; see also \cite[Section 4]{Harrison}}] \label{d_Hwang} 
    Let $F\in \operatorname{Sym}^dV^*$ and $g\in \operatorname{End}V$.
        \begin{enumerate}
            \item We define $F^g\in \otimes^dV^*$ by \[F^g(v_1,...,v_d)=F(gv_1,v_2,...,v_d), \text{ for } v_1,...,v_d\in V.\]
            \item  We say that $g$ is a \textbf{symmetrizer} of $F$ if $F^g$ belongs to $\operatorname{Sym}^dV^*$, namely, 
    \begin{align*}F(gv_1,v_2,...,v_d)&=F(v_1,gv_2,v_3,...,v_d)\\&=F(v_1,v_2,gv_3,...,v_d)\\&=\cdots\\&=F(v_1,...,v_{d-1},gv_d).\end{align*}
            \item The space of all symmetrizers of $F$ is denoted by $\symalg$, and it is called the \textbf{symmetrizer algebra} of $F$.
        \item The intersection $G_F=\symalg\cap \operatorname{GL}(V)$ is called the \textbf{symmetrizer group} of $F$.
        \end{enumerate}
\end{definition}

The symmetrizer algebra appeared implicitly in Mammana's work \cite{Mammana} and explicitly in Harrison's paper \cite[Section 4]{Harrison}, where it was referred to as the \emph{center}. It was later reintroduced by Hwang in \cite{Hwang}.

\begin{proposition}[{\cite[Proposition 2.2]{Hwang}}]\label{p_Hwang_2.2} Let $F\in \operatorname{Sym}^d V^*$.
\begin{enumerate}
    \item The vector space $\symalg$ is a Lie subalgebra of $\operatorname{End}(V)$ under the composition in $\operatorname{End}(V)$.
    \item If $F\in \operatorname{Sym}^d_oV^*$ then the Lie algebra $\symalg$ is abelian.
    \item The symmetrizer group $G_F$ is a connected closed subgroup of $\operatorname{GL}(V)$, corresponding to the Lie subalgebra $\symalg\subset \operatorname{End}(V)=\mathfrak{gl}(V)$. It is an abelian group if $F\in \operatorname{Sym}^d_oV^*$.
\end{enumerate}
\end{proposition}
    
    If $F\in \operatorname{Sym}^d_oV^*$, then Proposition \ref{p_Hwang_2.2} shows that $\symalg\subset \operatorname{End}(V)$ is an abelian Lie algebra. Moreover, since $\symalg$ is closed under composition in $\operatorname{End}(V)$, it may also be regarded as a commutative subalgebra of $\operatorname{End}(V)$. Throughout the paper, we will frequently use this algebra structure. The algebra $\symalg$ endowed with this structure is also called the \emph{centralizer}.

\begin{proposition}[{\cite[Proposition 2.3]{Hwang}}]\label{p_Hwang_2.3}
    For $F\in \operatorname{Sym}^dV^*$, the intersection 
    \[G_F=\symalg\cap \operatorname{GL}(V)\]
    is a connected closed subgroup of $\operatorname{GL}(V)$, corresponding to the Lie subalgebra $\symalg\subset \operatorname{End}(V)=\mathfrak{gl}(V)$. It is an abelian group if $F\in \operatorname{Sym}^d_oV^*$. 
\end{proposition}

Recall that for $F\in \operatorname{Sym}^d_oV^*$, $\symalg$ has the decomposition 
\[\symalg=\symalg^\times\oplus \symalg^+,\]
where $\symalg^+$ (resp. $\symalg^\times$) consists of nilpotent (resp. semisimple) elements.

The result \cite[Proposition 3.2 (ii)]{Hwang} can be generalized as follows.

\begin{proposition}\label{p_multiplicity}
    Let $F\in \operatorname{Sym}^dV^*$ and $g\in \symalg^+$. If $v\in V$ is an element such that $gv\ne 0$ and $g^rv=0$ with $2\le r\le d$, then $[gv]\in \mathbb{P}V$ is a point of $Z(F)$ with multiplicity at least $d-r+1$.
    \begin{proof}
        For any $v_1,...,v_{d-r}\in V$,
    \[F(\underbrace{g v,...,gv}_{r},v_{1},...,v_{d-r})=F(\underbrace{g^rv,v,...,v}_r,v_1,...,v_{d-r})=0.\]
    Thus $\operatorname{mult}_{[gv]}(F)\ge d-r+1$.
    \end{proof}
\end{proposition}

 The following two propositions allow us to compute the Jordan block forms of the symmetrizers.

\begin{proposition}\label{p_Jordan_blocks} Let $F\in \operatorname{Sym}^dV^*$. Let $g\in \symalg^+$ and let $k$ be the number of nonzero Jordan blocks of $g$. Then there is a $(k-1)$-dimensional linear subspace of $\mathbb{P}V$ whose every point is a singular point of $Z(F)$ with multiplicity $d-1$.
\begin{proof}
     Let 
    \[V\cong \left(\bigoplus_{i=1}^k \mathbb{C}[T]/(T^{r_i})\right)\oplus \mathbb{C}^q\]
    be the cyclic decomposition of $V$ with respect to the nilpotent element $g\in \operatorname{End}(V)$, where $r_1\ge\cdots\ge r_k\ge 2$. For each $i\in \{1,...,k\}$, let $\xi_i\in V$ be the element corresponding to $T^{r_i-2}\in \mathbb{C}[T]/(T^{r_i})$. Then $g\xi_1,...,g\xi_k\in V$ are linearly independent and $g^2\xi_i=0$. Thus, for any $(a_1:\cdots:a_k)\in \mathbb{P}^{k-1}$, we have
    \[g^2\left(\sum_{i=1}^ka_i\xi_i\right)=0.\]
    By Proposition \ref{p_multiplicity}, $[g(\sum_{i=1}^ka_i\xi_i)]\in \mathbb{P}V$ is a point of $Z(F)$ with multiplicity $d-1$. That is, every point of the $(k-1)$-dimensional linear subspace $\mathbb{P}\langle g\xi_1,...,g\xi_k\rangle$ is a singular point of $Z(F)$ with multiplicity $d-1$.
    \end{proof}
\end{proposition}

\begin{proposition}\label{p_power}
    Let $F\in \operatorname{Sym}^dV^*$. Assume that the locus of points in $Z(F)$ with multiplicity at least $d-1$ does not contain a linear subspace of dimension $k$. Then for any $g\in \symalg^+$, $g^{2k+1}=0$.
    \begin{proof}
        Assume the contrary that there is $g\in\symalg^+$ such that $g^{2k+1}\ne0$. Let $p>2k+1$ be the minimal integer such that $g^p=0$. Let $r=\lceil p/2\rceil$ and $h=g^r$. Then $h\in \mathfrak{g}_F^+$ is nonzero and $h^2=0$. Moreover, 
        \[\rk h=\rk g^r\ge p-r=\lfloor p/2\rfloor\ge k+1.\]
        By Proposition \ref{p_multiplicity}, every point in the $k$-dimensional linear space $\mathbb{P}(\im h)\subset \mathbb{P}V$ is a singular point of $Z(F)$ with multiplicity $d-1$. This is a contradiction.
    \end{proof}
\end{proposition}

The following proposition shows that, in order to bound $\dim G_F$, it suffices to bound $\dim_{\mathbb C}\symalg^+$ and the maximum of the ranks of elements in $\symalg^+$.

\begin{proposition}\label{p_inequality}
Let $F\in \operatorname{Sym}^d_oV^*$. For any $g\in \symalg^+$,
\[\dim G_F\le \dim_\mathbb{C}V- \rk g+\dim_\mathbb{C}\symalg^+.\]
\begin{proof} 
    We first show that if $A\subset \operatorname{End}(V)=\mathfrak{gl}(V)$ is a commutative subalgebra containing a nilpotent element of rank $r$, then the subspace of $A$ generated by diagonalizable elements in $A$ has dimension at most $\dim_\mathbb{C}V-r$.
    Let $\alpha\in A$ be a nilpotent element of rank $r$ and $\beta\in A$ be a diagonalizable element. Let $\lambda_1,...,\lambda_s$ be distinct eigenvalues of $\beta$ and consider the decomposition
        \[V=\bigoplus_{i=1}^s V_i,\]
    where $V_i$ is the eigenspace corresponding to $\lambda_i$. As $\beta$ commutes with $\alpha$, every $V_i$ is $\alpha$-invariant (i.e., $\alpha(V_i)\subset V_i$). So the restriction $\alpha|_{V_i}$ is a nilpotent operator on $V_i$, and $\rk(\alpha|_{V_i})\le \dim_\mathbb{C}V_i-1$. Thus 
        \[\sum_{i=1}^s\rk(\alpha|_{V_i})\le \sum_{i=1}^s(\dim_\mathbb{C}V_i-1)=\dim_\mathbb{C}V-s.\]
    Again, since $V_i$ are $\alpha$-invariant,
        \[\rk\alpha=\sum_{i=1}^s\rk(\alpha|_{V_i}).\]
        Thus $r=\rk\alpha\le \dim_\mathbb{C}V-s$, and therefore $s\le \dim_\mathbb{C}V-r$.
        That is, any diagonalizable element in $A$ has at most $\dim_\mathbb{C}V-r=n+1-r$ distinct eigenvalues. Now assume that there are linearly independent diagonalizable elements $\gamma_1,...,\gamma_{n+2-r}\in A$. As $A$ is abelian, $\gamma_1,...,\gamma_{n-r+2}$ are simultaneously diagonalizable. Thus one of the linear combinations of $\gamma_1,...,\gamma_{n-r+2}$ has at least $n-r+2$ distinct eigenvalues, which is absurd.

    Now consider $A=\symalg$. We have $\dim_\mathbb{C}\symalg^\times\le \dim_\mathbb{C}V-\rk g-1$, for any $g\in \symalg^+$. Hence
    \[\dim G_F=\dim_\mathbb{C}\symalg=(\dim_\mathbb{C}\symalg^\times+1)+\dim_\mathbb{C}\symalg^+\le \dim_\mathbb{C}V-\rk g+\dim_\mathbb{C}\symalg^+,\]
    for any $g\in \symalg^+$.
\end{proof}
\end{proposition}

\section{\bf Hessian of Hypersurfaces} \label{section_3}
In this section, we establish some basic properties of the Hessian of a symmetric $d$-form.

\begin{definition}\label{d_hessian}
    Let $F\in \operatorname{Sym}^dV^*$. The \textbf{Hessian} of $F$ is the symmetric pairing 
    \[\mathcal{H}_F:V\times V\to \operatorname{Sym}^{d-2}V^*\]
    given by 
    \[\mathcal{H}_F(u,w)=\partial_u\partial_w F=\partial_w\partial_u F.\]
    This pairing induces the natural homomorphism \begin{align*}
        h_F:V\to  \operatorname{Hom}_\mathbb{C}(V,\operatorname{Sym}^{d-2}V^*),
    \end{align*}
    given by $h_F(w)=\partial_w\partial_uF$.
\end{definition}

\begin{proposition} \label{p_equivalent}
    Let $g\in \operatorname{End}(V)$ and $F\in \operatorname{Sym}^d V^*$. The following are equivalent.
    \begin{enumerate}
        \item The element $g$ is a symmetrizer of $F$.
        \item For any $u,w\in V$, $\mathcal{H}_F(gu,w)=\mathcal{H}_F(u,gw)$.
    \end{enumerate}
    \begin{proof} Recall that for $u\in V$, $\partial_uF$ denotes the symmetric $(d-1)$-linear form 
    \[\partial_uF(v_1,...,v_{d-1})=F(u,v_1,...,v_{d-1}).\]
    Thus, for $u,w\in V$, $\mathcal{H}_F(u,w)$ is the symmetric $(d-2)$-form given by
    \[\mathcal{H}_F(u,w)(v_1,...,v_{d-2})=(\partial_u\partial_wF)(v_1,...,v_{d-2})=F(u,w,v_1,...,v_{d-2}).\]
        $(1)\implies (2)$: Assume that $g\in \symalg$. Then for any $u,w\in v$ and $v_1,...,v_{d-2}\in V$, 
        \begin{align*}
        \mathcal{H}_F(gu,w)(v_1,...,v_{d-2})&=(\partial_{gu}\partial_wF)(v_1,...,v_{d-2})\\
        &=F(gu,w,v_1,...,v_{d-2})\\
        &=F(u,gw,v_1,...,v_{d-2})\\
        &=(\partial_u\partial_{gw}F)(v_1,...,v_{d-2})\\
        &=\mathcal{H}_F(u,gw)(v_1,...,v_{d-2}).
        \end{align*} Hence $\mathcal{H}_F(gu,w)=\mathcal{H}_F(u,gw)$ for any $u,w\in V$.
        
        $(2)\implies (1)$: Let $g\in \operatorname{End}(V)$ be an element such that $\mathcal{H}_F(gu,w)=\mathcal{H}_F(u,gw)$ for all $u,w\in V$. Let $v_1,...,v_d\in V$. Since $\mathcal{H}_F(gv_1,v_2)=\mathcal{H}_F(v_1,gv_2)$,
        \begin{align*}
            F(gv_1,v_2,...,v_d)
            &=(\partial_{gv_1}\partial_{v_2}F)(v_3,...,v_d)\\
            &=\mathcal{H}_F(gv_1,v_2)(v_3,...,v_d)\\
            &=\mathcal{H}_F(v_1,gv_2)(v_3,...,v_d)\\
            &=(\partial_{v_1}\partial_{gv_2})F(v_3,...,v_d)\\
            &=F(v_1,gv_2,v_3,...,v_d).
        \end{align*} 
            Since $F$ is a symmetric form, and since $\mathcal{H}_F(gv_2,v_3)=\mathcal{H}_F(v_2,gv_3)$, we have
        \begin{align*}
            F(v_1,v_2,gv_3,...,v_d)&=F(gv_2,v_3,v_1,v_4,...,v_d)\\
            &=\mathcal{H}_F(gv_2,v_3)(v_1,v_4,...,v_d)\\
            &=\mathcal{H}_F(v_2,gv_3)(v_1,v_4,...,v_d)\\
            &=F(v_2,gv_3,v_1,v_4,...,v_d)\\
            &=F(v_1,v_2,gv_3,v_4,...,v_d).
        \end{align*}
        Repeating this process shows that
        \[F(gv_1,v_2,...,v_d)=F(v_1,gv_2,v_3,...,v_d)=F(v_1,v_2,gv_3,...,v_d)=\cdots=F(v_1,v_2,...,gv_d).\]
        That is, $g\in \symalg$.
    \end{proof}
\end{proposition}

\begin{remark}\label{r_Hessian_matrix}
    Fixing a basis $\mathfrak{B}=\{\xi_0,...,\xi_n\}$ of $V$, we obtain the canonical isomorphism $\theta_i:\operatorname{Sym}^iV^*\to \mathbb{C}[x_0,...,x_n]_i$ for each $i\in \mathbb{Z}^{\ge 0}$. Let $P=\theta_d(F)$ be the homogeneous polynomial corresponding to $F\in \operatorname{Sym}^dV^*$.
    Then the matrix
    \[\left(\theta_{d-2}(\mathcal{H}_F(\xi_i,\xi_j))\right)_{0\le i,j\le n}\in \mathbb{C}[x_0,...,x_n]^{(n+1)\times (n+1)}\]
    is a nonzero constant multiple of the Hessian matrix
    \[H_P=(\partial^2P/\partial x_i\partial x_j)_{0\le i,j\le n}\]
    of $P$. Thus, if $g\in \operatorname{End}(V)$ and if $A$ is the matrix representation of $g$ with respect to $\mathfrak{B}$, then by Proposition \ref{p_equivalent}, $g\in \symalg$ if and only if $H_P\cdot A$ is symmetric if and only if 
    \[A^t\cdot H_P=H_P\cdot A\in \mathbb{C}[x_0,...,x_n]^{(n+1)\times (n+1)}.\] 
    This equation already appeared in \cite[page 126]{Mammana} and \cite[page 24]{CG79}.
\end{remark}

\begin{proposition} The function $\operatorname{Sym}^d_oV^*\to \mathbb{Z}$ given by \begin{align*}
    F&\mapsto \dim_\mathbb{C}\symalg^+
\end{align*} is upper semicontinuous.
    \begin{proof}
        We work with coordinates (see Remark \ref{r_Hessian_matrix}). Let $S_d=\mathbb{C}[x_0,...,x_n]_d$. Let $\mathfrak{N}\subset \mathbb{P} (\mathbb{C}^{(n+1)\times (n+1)})$ be the closed subscheme consisting of nilpotent matrices. Let \[Z=\{(P,A)\in \mathbb{P} (S_d)\times \mathfrak{N}\mid H_P\cdot A \text{ is symmetric}\},\]
        where $H_P$ is the Hessian matrix of $P$. Clearly, $Z$ is a closed subscheme of $\mathbb{P} (S_d)\times \mathfrak{N}$ cut out by polynomials that are homogeneous in both of the coordinates of $\mathbb{P} (S_d)$ and $\mathbb{P} (\mathbb{C}^{(n+1)\times (n+1)})$. The projection map $\pi:Z\to \mathbb{P} (S_d)$ is a proper morphism, hence the function $P\mapsto \dim \pi^{-1}(P)$ is upper semicontinuous (see Theorem 12.4.3 of \cite{Vakil}). If $F\in\operatorname{Sym}^d_oV^*$ and $P\in S_d$ is the corresponding polynomial, then $\pi^{-1}(P)=\mathbb{P}\symalg^+$ (which is the empty set if $\symalg^+=0$). Thus we conclude that
        \[F\mapsto \dim \pi^{-1}(P)+1=\dim_\mathbb{C}\symalg^+\] 
        is upper semicontinuous.
    \end{proof}
\end{proposition}

\begin{definition} \label{d_r} Let $F\in \operatorname{Sym}^dV^*$ and $\xi V$. Consider the linear map 
\[h_F(\xi):V\to \operatorname{Sym}^{d-2}V^*\] 
induced by $\mathcal{H}_F$. The \textbf{rank} of $\xi$ with respect to $F$ is \begin{align*}
    \operatorname{rank}_F(\xi):=\dim_\mathbb{C}\operatorname{im}h_F(\xi).
\end{align*}
\end{definition}

\begin{remark} Let $[\xi]\in Z(F)$ be a point with multiplicity $\ge d-1$.
\begin{enumerate}
    \item Since $0\ne \xi\in \ker h_F(\xi)$, we have \[0\le\operatorname{rank}_F(\xi)=\dim_\mathbb{C}\im h_F(\xi)\le \dim_\mathbb{C}V-1=n.\]
    \item  Suppose that $F\in \operatorname{Sym}^d_oV^*$. Then $ h_F(\xi)$ is not identically zero and so 
    \[1\le \operatorname{rank}_F(\xi)=\dim _\mathbb{C}\im h_F(\xi)\le n.\]
\end{enumerate}
\end{remark}

\begin{remark}\label{r_rank_polynomial} The rank of a point can be described in terms of polynomials as follows. Let $P\in \mathbb{C}[x_0,...,x_n]_d$. \begin{enumerate} 
    \item Let $H_P$ be the Hessian matrix of $P$. $H_P$ defines the symmetric pairing \begin{align*}\mathcal{H}_P:\mathbb{C}^{n+1}\times \mathbb{C}^{n+1}&\to \mathbb{C}[x_0,...,x_n]_{d-2},\\ (u,w)&\mapsto u^t\cdot H_P\cdot w.\end{align*} For each $u\in \mathbb{C}^{n+1}$, we have the induced homomorphism $h_P(u):\mathbb{C}^{n+1}\to \mathbb{C}[x_0,...,x_n]_{d-2}$. The rank of a point $[\xi]\in \mathbb{P}^n$ is $\operatorname{rank}_P(\xi):=\dim \im h_P(\xi)$.  
    \item Let $P\in \mathbb{C}[x_0,...,x_n]$ be a homogeneous polynomial of degree $d$ and $[\xi]\in Z(P)\subset \mathbb{P}^n$ be a point of multiplicity $\ge d-1$, where $\xi=(1,0,...,0)$. Then $P=x_0Q+R$ for some homogeneous $Q,R\in \mathbb{C}[x_1,...,x_n]$ with $\deg Q=d-1$ and $\deg R=d$. Computing $H_P$, we see that $\operatorname{im} h_P(\xi)$ is equal to the degree $d-1$ component $(J_Q)_{d-1}$ of the Jacobian ideal $J_Q = (\partial Q/\partial x_1, \dots, \partial Q/\partial x_n) \subset \mathbb{C}[x_1, \dots, x_n]$ of $Q$.
    \item From (2), it follows that if $[\xi]$ is a point of the hypersurface $Z(P)\subset \mathbb{P}^n$ with multiplicity $d-1$ and $\operatorname{rank}_P(\xi)=r$ then the tangent cone to $Z(P)$ at $[\xi]$ is an affine hypersurface in $\mathbb{C}^n$ cut out by a homogeneous polynomial $Q$ of degree $d-1$ such that $\dim_\mathbb{C}(J_Q)_{d-1}=r$. Thus geometrically, $\operatorname{rank}_P(\xi)$ measures non-degeneracy of the tangent cone at $[\xi]$.
\end{enumerate}
\end{remark}

The following proposition allows us to measure the rank of points in the image of a symmetrizer.

\begin{proposition} \label{p_r}
    Let $F\in \operatorname{Sym}^dV^*$ and $g\in \symalg$. For any $\xi\in V$, \[\dim_\mathbb{C}\operatorname{im}h_F(g\xi)=\rk g-\dim_\mathbb{C}(\operatorname{im} g\cap \ker h_F(\xi)).\] In particular, $\dim_\mathbb{C}\im h_F(g\xi)\le \rk g$.
    \begin{proof}
            For $u\in V$, $h_F(g\xi)(u)=0$ if and only if $h_F(\xi)(gu)=0$ by Proposition \ref{p_equivalent}. Thus 
            \[g(\ker h_F(g\xi))=\operatorname{im}g\cap \ker h_F(\xi).\] 
            On the other hand, if $w\in \ker g$ then $h_F(g\xi)(w)=h_F(\xi)(0)=0$, hence 
            \[\ker g\subset \ker h_F(g\xi).\] 
            Therefore, we have an exact sequence
            \[0\to \ker g\to\ker h_F(g\xi)\xrightarrow{g} \im g\cap \ker h_F(\xi)\to 0.\]
            Hence 
            \[\dim_\mathbb{C}(\im g\cap \ker h_F(\xi))
            =\dim_\mathbb{C}\ker h_F(g\xi)-\dim_\mathbb{C}\ker g.\]
            Subtracting both sides from $\dim_\mathbb{C}V$ gives the required result.
        \end{proof} 
\end{proposition}

The proof of the following proposition is straightforward, but the result serves as a foundation for the proofs of Theorems \ref{t_main}, Theorem \ref{t_main_2}, and Theorem \ref{t_quasi-vertices}.

\begin{proposition} \label{p_kernel}
    Let $k\ge 1$ and $F\in \operatorname{Sym}^dV^*$. Let $g,h\in \symalg^+$ be elements with rank $k$ such that $g^2=h^2=0$. Assume that and that there is an element $\xi\in\im g\cap \im h$ such that $\operatorname{rank}_F(\xi)=k$. Then $\ker g=\ker h$.
    \begin{proof}
        Since $\xi \in \im g\cap \im h$ and $g^2=h^2=0$, we have $\ker g\subset \ker h_F(\xi)$ and $\ker h\subset \ker h_F(\xi)$. As $\operatorname{rank}_F(\xi)=k$, we have $\dim_\mathbb{C}\ker g=\dim_{\mathbb{C}}\ker h_F(\xi)=\dim_\mathbb{C}\ker h$. Hence $\ker g=\ker h_F(\xi)=\ker h$.
    \end{proof}
\end{proposition}

We close this section by showing that the presence of a certain type of singularity forces $\symalg^+$ to vanish.

\begin{corollary}\label{c_r=n}
Let $F\in \operatorname{Sym}^d_o V^*$. If there is a singular point $[\xi]\in Z(F)$ with multiplicity $d-1$ such that $\operatorname{rank}_F(\xi)=n$, then $\symalg^+=0$.
    \begin{proof} Since $\operatorname{rank}_F(\xi)=n$, $\dim_\mathbb{C}\ker h_F(\xi)=1$. As $\operatorname{mult}_{[\xi]}(F)=d-1$, $\xi\in\ker h_F(\xi)$. i.e., $\ker h_F(\xi)=\langle \xi\rangle$. Suppose that there is a nonzero element $h\in \symalg^+$. Then for some $l\ge 1$, $g:=h^l$ satisfies $g\ne 0$ and $g^2=0$. Since we are assuming that $\dim_\mathbb{C}V=n+1\ge 3$, $n>(n+1)/2$. As $g^2=0$, $(n+1)/2\ge \rk g$. Thus 
    \[\operatorname{rank}_F(\xi)=n>(n+1)/2\ge \rk g.\]
    The inequality of Proposition \ref{p_r} shows that $\xi\notin \im g$. i.e., 
    \[\im g\cap \ker h_F(\xi)=0.\]
    Since 
    \[
    F(\xi, g\xi, v_1, \dots, v_{d-2}) = F(\xi, \xi, gv_1, \dots, v_{d-2}) = 0
    \]
    for all $v_1, \dots, v_{d-2} \in V$, we have $g\xi \in \ker h_F(\xi) = \langle \xi \rangle$. But $\xi \notin \operatorname{im} g$, so $g\xi = 0$. By Proposition \ref{p_r},
    \[
    \rk g= \dim_{\mathbb{C}} \operatorname{im} h_F(g\xi) = 0,
    \]
which contradicts $g \ne 0$. Therefore, $\symalg^+ = 0$.
    \end{proof}
\end{corollary}

\begin{example}
    Let $Z(F)\subset \mathbb{P} V$ be a cubic hypersurface. By Remark \ref{r_rank_polynomial}, if $[\xi]\in Z(F)$ is a nodal singularity (i.e., a singularity locally analytically equivalent to the origin of $(x_1^2+\cdots+x_n^2=0)\subset \mathbb{C}^n$) then 
    \[\operatorname{rank}_F(\xi)=\dim_\mathbb{C}\langle x_1,...,x_n \rangle=n.\] 
    By Corollary \ref{c_r=n}, $\symalg^+=0$ whenever $Z(F)$ has a nodal singularity.
\end{example}

\section{\bf Quasi-vertices}\label{section_4}

In this section, we investigate the behavior of quasi-vertices. Throughout this section, we fix a symmetric form $F \in \operatorname{Sym}^d V^*$. Before proceeding, we remark that the case $m=0$ of the following theorem of Mammana implies that the symmetrizer group of a hypersurface $Z(F)$ admitting a quasi-vertex is nontrivial.

\begin{theorem}[{\cite[page 134]{Mammana}}]\label{t_Mammana}
    Let $F\in \operatorname{Sym}^d_oV^*$ and $P\in \mathbb{C}[x_0,...,x_n]$ be the corresponding homogeneous polynomial. Then $G_F\ne \mathbb{C}\cdot \operatorname{Id}_V$ if and only if either
    \begin{itemize}
        \item $P$ is projectively equivalent to a polynomial $P_1+P_2$ where $P_1\in \mathbb{C}[x_0,...,x_r]_d$ and $P_2\in \mathbb{C}[x_{r+1},...,x_n]_d$ with $0\le r<n$, or
        \item $P$ is projectively equivalent to a polynomial of the form
        \[\sum_{i=0}^{m}x_i\frac{\partial Q}{\partial x_{i+m+1}}+R,\]
        where $Q(x_{m+1},...,x_{2m+1})$, $R(x_{m+1},...,x_n)$ are homogeneous of degree $d$ and $m\le (n-1)/2$.
    \end{itemize}
\end{theorem}

We begin by describing quasi-vertices in terms of the notion $\operatorname{rank}_F(\xi)$.

\begin{lemma}\label{l_linear_change_r}\begin{enumerate}
    \item Let $P\in \mathbb{C}[x_0,...,x_m]$ be a homogeneous polynomial and $r$ the dimension of the vector space generated by the partial derivatives of $P$. After a linear change of coordinates, $P$ is a polynomial in $r$ variables. 
    \item Let $[\xi]\in Z(F)$ be a point such that $\operatorname{rank}_F(\xi)\le 1$. Then $\operatorname{mult}_{[\xi]}(F)\ge d-1$.
\end{enumerate} 
    \begin{proof}
     (1) As this is a well-known fact, we only sketch the proof (see Proposition 1.1.6 of \cite{Dolgachev}). If $r=m+1$, we are done. If not, $Z(P)\subset \mathbb{P}^{m}$ is a cone over a hypersurface in $\mathbb{P}^{m-1}$. So after a linear change of coordinates, $P$ is a polynomial in $m$ variables. Repeating this completes the proof.

     (2) Suppose not. As $\operatorname{mult}_{[\xi]}(F)\ge d-1$ if and only if $\xi\in \ker h_F(\xi)$, we have $\xi\notin \ker h_F(\xi)$ and so $V=\ker h_F(\xi)\oplus\langle \xi\rangle$. Let $v_1,...,v_{d-2}\in V$ be elements such that \[F(\xi,\xi,v_1,...,v_{d-2})\ne0.\] We can write $v_i=u_i+c_i\xi$ for some $u_i\in \ker h_F(\xi)$ and $c_i\in \mathbb{C}$, so that 
        \begin{align*}
            0\ne F(\xi,\xi,v_1,...,v_{d-2})=c_1\cdots c_{d-2} F(\xi,...,\xi).
        \end{align*}
     However, since $[\xi]\in Z(F)$, $F(\xi,...,\xi)=0$ . This is a contradiction.
    \end{proof}
\end{lemma}

\begin{lemma}\label{l_r_iff_quasi-vertex}
    Let $F\in \operatorname{Sym}^dV^*$ and $[\xi]\in Z(F)$.
    \begin{enumerate}
        \item The point $[\xi]$ is a quasi-vertex of $Z(F)$ if and only if $\operatorname{rank}_F(\xi)=1$.
        \item The hypersurface $Z(F)$ is a cone with vertex $[\xi]$ if and only if $\operatorname{rank}_F(\xi)=0$.
    \end{enumerate}
    \begin{proof}
        (1) We work with the projective coordinate $(x_0:\cdots:x_n)$ of $\mathbb{P}^n=\mathbb{P} V$ and with the homogeneous polynomial $P\in \mathbb{C}[x_0,...,x_n]$ of degree $d$, corresponding to $F$. We may assume that $[\xi]=(1:0:\cdots:0)$. By Lemma \ref{l_linear_change_r}, $[\xi]$ is a singularity of multiplicity $\ge d-1$, so we can write 
        \[P=x_0Q+R\]
        for some $Q\in \mathbb{C}[x_1,...,x_n]_{d-1}$ and $R\in \mathbb{C}[x_1,...,x_n]_d$. Moreover, 
        \[\operatorname{rank}_F(\xi)=\dim_\mathbb{C}\langle \partial Q/\partial x_1,...,\partial Q/\partial x_n\rangle.\]
        Assume that $\operatorname{rank}_F(\xi)=1$. By Lemma \ref{l_linear_change_r}(1), $Q=L^{d-1}$ for some nonzero linear form $L$. Thus $[\xi]$ is a quasi-vertex. Conversely, assume that $[\xi]$ is a quasi-vertex. Then $Q=L^{d-1}$ for some nonzero linear form $L$, so \[\operatorname{rank}_F(\xi)=\dim_\mathbb{C}\langle\partial L^{d-1}/\partial x_1,...,\partial L^{d-1}/\partial x_n \rangle=1.\]

        (2) The hypersurface $Z(F)$ is a cone with vertex $[\xi]$ if and only if $\operatorname{mult}_x(F)=d$ if and only if $\ker h_F(\xi)=V$ if and only if $\operatorname{rank}_F(x)=0$.
    \end{proof}
\end{lemma}

\begin{proposition} \label{p_quasi-vertex_mult} If $[\xi]\in Z(F)$ is a quasi-vertex and $[\eta]\in Z(F)$ is a point of multiplicity $m\ge 2$, then every point on the line joining $[\xi]$ and $[\eta]$ is a singular point of $Z(F)$ with multiplicity at least $m$.
\begin{proof}
   Let $L=\mathbb{P}\langle \xi,\eta\rangle\subset \mathbb{P}V$ and $[a\xi+b\eta]\in L$, where $(a:b)\in \mathbb{P}^1$. For any $v_1,...,v_{m-1}\in V$,
        \begin{align*}
        &F(\underbrace{a\xi+b\eta,...,a\xi+b\eta}_{d-m+1},v_1,...,v_{m-1})\\
        &=\sum_{i=0}^{d-m+1}\binom{d-m+1}{i}a^ib^{d-m+1-i}F(\underbrace{\xi,...,\xi}_{i},\underbrace{\eta,...,\eta}_{d-m+1-i},v_1,...,v_{m-1})\\
        &=(d-m+1)ab^{d-m}F(\xi,\underbrace{\eta,...,\eta}_{d-m},v_1,...,v_{m-1})+b^{d-m+1}F(\underbrace{\eta,...,\eta}_{d-m+1},v_1,...,v_{m-1})\\
        &=(d-m+1)ab^{d-m}F(\xi,\underbrace{\eta,...,\eta}_{d-m},v_1,...,v_{m-1}).
        \end{align*} If $\eta\in \ker h_F(\xi)$, then this vanishes. Suppose that $\eta\notin \ker h_F(\xi)$. By Lemma \ref{l_r_iff_quasi-vertex}, $\operatorname{rank}_F(\xi)=1$. Thus $\dim_\mathbb{C}\ker h_F(\xi)= n$ and $V=\ker h_F(\xi)\oplus \langle \eta\rangle$. Write $v_1=u+c\eta$ for some $u\in \ker h_F(\xi)$ and $c\in \mathbb{C}$. As $\operatorname{mult}_{[\eta]}(F)=m$, 
        \begin{align*}
            F(\xi,\eta,...,\eta,v_1,...,v_{m-1})&=F(\xi,\eta,...,\eta,u,v_2,...,v_{m-1})+F(\xi,\eta,...,\eta,c\eta,v_2,...,v_{m-1})\\
            &=cF(\eta,...,\eta,\eta,\underbrace{\xi,v_2,...,v_{m-1}}_{m-1})=0.
        \end{align*} Thus every point on $L$ has multiplicity at least $m$ in $Z(F)$.
\end{proof}
\end{proposition}

This immediately implies:

\begin{corollary}
    Let $X\subset \mathbb{P}^n$ be a hypersurface with isolated singularities. If $X$ has a quasi-vertex, then it is the unique singular point of $X$.
\end{corollary}

Now we prove Theorem \ref{t_quasi-vertices}. For this, we require one more lemma, which can also be derived from Mammana's proof of Theorem \ref{t_Mammana}.

\begin{lemma}\label{l_r at most 1} 
    Let $[\xi_0]\in Z(F)$. Then there is a nilpotent $g\in \symalg^+$ such that $\im g=\langle \xi_0\rangle$ if and only if  $\operatorname{rank}_F(\xi_0)\le 1$.
 \begin{proof}
     Assume that $\im g=\langle \xi_0\rangle$ for some nilpotent $g\in \symalg$ of rank $1$. As $\rk g=1$, by Proposition \ref{p_r}, $\operatorname{rank}_F(\xi_0)\le 1$. Conversely, assume that $\operatorname{rank}_F(\xi_0)\le 1$. Then $\dim_\mathbb{C}\ker h_F(\xi_0)\ge n$. By Lemma \ref{l_linear_change_r}(2), $\operatorname{mult}_{[\xi_0]}(F)\ge d-1$ and so $\xi_0\in \ker h_F(\xi_0)$. Thus, there is a basis $\{\xi_0,...,\xi_n\}$ of $V$ such that $\xi_0,...,\xi_{n-1}\in \ker h_F(\xi_0)$. Consider the endomorphism 
     \[g:=\xi_0\otimes \xi_n^*:V\to V,\quad v\mapsto \xi_n^*(v)\cdot \xi_0,\] where $\xi_n^*$ is the dual vector of $\xi_n$. We have $g^2=0$ and $\im g=\langle \xi_0\rangle$. We now show that $g\in \symalg$. By Proposition \ref{p_equivalent}, it is enough to show that 
    \[\mathcal{H}_F(g\xi_i,\xi_j)=\mathcal{H}_F(\xi_i,g\xi_j), \text{ for }i,j=0,...,n.\]
    If one of $i,j$ is at most $n-1$, then both sides of the equation are $0$. If $i=j=n$, then the equality holds because $\mathcal{H}_F$ is symmetric. Hence $g\in \symalg$.
 \end{proof}
\end{lemma}

\smallskip

\begin{proof}[Proof of Theorem \ref{t_quasi-vertices}]
    (1) Let $F\in \operatorname{Sym}^d_oV^*$. Our aim is to show that the map
    \begin{align*}
    \{g\in \mathbb{P}\symalg^+\mid \rk g=1\}&\to \{\text{quasi-vertices of }Z(F)\}\\
    g&\mapsto \mathbb{P}(\im g)
    \end{align*} is bijective. Note that this map is well-defined, since if $g\in \symalg^+$ has rank $1$, then $\mathbb{P}(\im g)$ consists of one point $[\xi]$ such that $\operatorname{rank}_F(\xi)=1$. The surjectivity follows immediately from Lemma \ref{l_r_iff_quasi-vertex} and Lemma \ref{l_r at most 1}: as $Z(F)$ is not a cone, $[\xi]\in Z(F)$ is a quasi-vertex if and only if $\operatorname{rank}_F(\xi)=1$ if and only if $[\xi]=\mathbb{P} (\im g)$ for some $g\in \symalg^+$. 
    
    To show that the map is injective, suppose that $g,h\in \symalg^+$ are two elements of rank $1$ such that $\mathbb{P}(\im g)=\mathbb{P}(\im h)$. Write $\im g=\im h=\langle \xi\rangle$. By Proposition \ref{p_r} and Lemma~\ref{l_r_iff_quasi-vertex}, $\operatorname{rank}_F(\xi)=1$. As $g$ and $h$ are nilpotent with rank 1, $g^2=h^2=0$. By Proposition \ref{p_kernel}, $\ker g=\ker h$. Hence $g$ and $h$ are proportional. This proves (1).
    
    (a)$\implies$(b): Choose any $0\ne h\in \symalg^+$ and take its power $g:=h^{\ell}$ so that $g\ne 0$ and $g^2=0$. By Proposition \ref{p_multiplicity}, every point in $\mathbb{P}(\im g)$ is a singularity of $Z(F)$ with multiplicity $d-1$. Since the locus of points in $Z(F)$ with multiplicity $d-1$ does not contain a line, $\rk g=1$. By (1), $[\xi]:=\mathbb{P}(\im g)$ is a quasi-vertex of $Z(F)$. If there is another point $[\eta]\ne [\xi]$ in $Z(F)$ with multiplicity $d-1$, then by Proposition \ref{p_quasi-vertex_mult}, the line joining $[\xi]$ and $[\eta]$ has multiplicity $d-1$. Hence $[\xi]$ is the unique point of multiplicity $d-1$.
    
    (b)$\implies$(c): Trivial.
    
    (c)$\implies$(a): This follows from (1). This proves (2).
\end{proof}

The following result will be needed in the proof of Theorem \ref{t_main_2}.

 \begin{proposition}\label{p_quasi-vertices_conic}
    Let $F\in \operatorname{Sym}^d_oV^*$. A line in $Z(F)$ contains at most two quasi-vertices of $Z(F)$.
    \begin{proof}
        We work with coordinates. Let $P$ be the corresponding polynomial. Say $x=(1:0:\cdots:0)$ and $y=(0:1:0:\cdots:0)$ are quasi-vertices. Then we can write 
        \[P=x_0L_0^{d-1}+x_1L_1^{d-1}+R(x_2,...,x_n)\]
        for some linear forms $L_i\in \mathbb{C}[x_2,...,x_n]$ and homogeneous $R$ of degree $d$. As $Z(F)$ is not a cone, $L_0,L_1$ are linearly independent. Consider a point $z=(1:t:0:\cdots:0)$, where $t\in \mathbb{C}^\times$. Taking the coordinate change $x_1\mapsto tx_0+x_1$, we see that the tangent cone at $z$ is cut out by $L_0^{d-1}+tL_1^{d-1}$ and this must be a power of a linear form if $z$ is a quasi-vertex. As $L_0, L_1$ are linearly independent linear forms, this is impossible.
    \end{proof}
\end{proposition}

\begin{remark}
    There exists a hypersurface $Z(F)$ that is not a cone and has infinitely many quasi-vertices. Consider the cubic threefold $X\subset \mathbb{P} ^4$ defined by $P=x_0x_3^2+x_1x_4^2+x_2(x_3+x_4)^2$ and the corresponding symmetric $d$-form $F$. As the partial derivatives of $P$ are linearly independent, $X$ is not a cone. The locus of quasi-vertices of $Z(F)$ is a smooth plane conic: every point $(a:b:c:0:0)\in \mathbb{P} ^4$ such that $ab+bc+ac=0$ is a quasi-vertex of $X$.
\end{remark}

\begin{proposition}\label{p_quasi-vertex_stabilized}
    Let $F\in\operatorname{Sym}^d_oV^*$.
    \begin{enumerate}
        \item If $[\xi]\in Z(F)$ is a quasi-vertex, then $\xi$ is an eigenvector of $g$ for any $g\in \symalg$. In particular, $G_F$ stabilizes $[\xi]$: $[g\xi]=[\xi]$ for all $g\in G_F$.
        \item If $g\in \symalg^+$ is an element of rank $1$, then for any $f\in \symalg^+$, $fg=0$.
    \end{enumerate}
    \begin{proof}
        (1) Let $g\in \symalg$. Choose general $(a:b)\in \mathbb{P}^1$ so that the symmetrizer
        \[\varphi_{a,b}:=a\cdot\operatorname{Id}_V+bg\]
        is invertible. By Proposition \ref{p_r},
        \[\operatorname{rank}_F(\varphi_{a,b}(\xi))=\dim_\mathbb{C}V-\dim_\mathbb{C}\ker h_F(\xi)=1.\]
        By Lemma \ref{l_r_iff_quasi-vertex}(1), $\varphi_{a,b}(\xi)=a\xi+bg\xi$ is a quasi-vertex. So if $g\xi \notin \langle \xi \rangle$, then general points of the line $\mathbb{P}\langle \xi, g\xi \rangle \subset \mathbb{P}V$ are quasi-vertices of $Z(F)$. But this is impossible by Proposition~\ref{p_quasi-vertices_conic}.

        (2) Let $g\in \symalg^+$ be an element of rank $1$ and write $\im g=\langle v\rangle$. By Theorem \ref{t_quasi-vertices}, $[v]$ is a quasi-vertex of $Z(F)$. Let $f\in \symalg^+$ be any element. Since $f$ is nilpotent, $fv=0$ by (1). Thus $fg=0$.
    \end{proof}
\end{proposition}

 The vertex $z$ of a cone $Z\subset \mathbb{P}^n$ can be characterized as the singularity such that the locus of lines on $Z$ passing through $z$ is $Z$ itself. A quasi-vertex has an analogous characterization:

\begin{proposition} \label{p_quasi-vertex_lines through x} Let $[\xi]\in Z(F)\subset \mathbb{P}V$ and $n\ge 3$. Then $[\xi]$ is a quasi-vertex of $Z(F)\subset \mathbb{P}V$ if and only if the locus $\Lambda$ of lines in $Z(F)$ passing through $[\xi]$ is a (set-theoretic) hyperplane section of $Z(F)$.
    \begin{proof}
         Assume that $[\xi]$ is a quasi-vertex of $Z(F)$. By Lemma \ref{l_r_iff_quasi-vertex}(1), $\operatorname{rank}_F(\xi)=1$ and so $\ker h_F(\xi)$ is $n$-dimensional. i.e., $\mathbb{P} (\ker h_F(\xi))\subset \mathbb{P} V$ is a hyperplane. Consider the hyperplane section $D:=\mathbb{P} (\ker h_F(\xi))\cap Z(F)$. We show that $D=\Lambda$. Let $[\eta]\in D$ be a point distinct from $[\xi]$. Since $\eta\in \ker h_F(\xi)$ and $[\eta]\in Z(F)$,
        \begin{align*}
            F(a\xi+b\eta,...,a\xi+b\eta)&=\sum_{i=0}^d\binom{d}{i}a^ib^{d-i}F(\underbrace{\xi,...,\xi}_{i},\underbrace{\eta,...,\eta}_{d-i})\\
            &=d\cdot ab^{d-1}F(\xi,\eta,...,\eta)+b^dF(\eta,...,\eta)=0,
        \end{align*} for all $a,b\in \mathbb{C}$. Hence the whole line $\mathbb{P} \langle \xi,\eta\rangle$ is contained in $Z(F)$, and in particular, $[\eta]\in \Lambda$. Conversely, let $L$ be a line in $Z(F)$ passing through $[\xi]$. Suppose that there is a point $[\eta]\in L$ that is not in $D$. Since $F(\xi+\eta,...,\xi+\eta)=0$, since $\mathcal{H}_F(\xi,\xi)=0$ and $F(\eta,...,\eta)=0$, we have $F(\xi,\eta,...,\eta)=0$. As $\eta\notin \ker h_F(\xi)$, there are $v_1,...,v_{d-2}\in V$ such that $F(\xi,\eta,v_1,...,v_{d-2})\ne 0$. As $\operatorname{rank}_F(\xi)=1$, we have $V=\ker h_F(\xi)\oplus \langle \eta \rangle$, so we can write $v_i=u_i+c_i\eta$ for some constants $c_i$. Thus 
        \begin{align*}F(\xi,\eta,v_1,...,v_{d-2})
        =c_1\cdots c_{d-2}F(\xi,\eta,\eta,...,\eta)=0,
        \end{align*} a contradiction. Hence $L\subset D$ and $D=\Lambda$.

        For the converse, assume that $H=\mathbb{P} W\subset \mathbb{P} V$ is a hyperplane such that $\Lambda=H\cap Z(F)$ is the locus of lines passing through $[\xi]$, where $W\subset V$ is an $n$-dimensional subspace. Let $G:=F|_{W}$ be the restriction of $F$. As the locus of lines in $\Lambda=Z(G)\subset \mathbb{P} W$ passing through $[\xi]$ is $\Lambda$ itself, $\Lambda$ is a cone with vertex $[\xi]$. Thus $W=\ker h_G(\xi)\subset \ker h_F(\xi)$. If $Z(F)$ is a cone with vertex $[\xi]$ then the locus of lines in $Z(F)$ passing through $Z(F)$ must be $Z(F)$. But as the locus is a hyperplane section, $\dim_\mathbb{C}\ker h_F(\xi)<\dim_\mathbb{C}V$, hence $\ker h_F(\xi)=W$. i.e., $\operatorname{rank}_F(x)=1$. We are done by Lemma \ref{l_r_iff_quasi-vertex}(1).
    \end{proof}
\end{proposition}

\section{\bf{Dimension Bounds}} \label{section_5}

In this section, we prove Theorem \ref{t_main} and Theorem \ref{t_main_2}. We begin by establishing several foundational results that will serve as the basis for the proof. Throughout the section, we fix a symmetric $d$-form $F\in \operatorname{Sym}^d_oV^*$ and consider the Artinian local ring 
\[R_F:=\mathbb{C}\oplus\symalg^+\]
with the maximal ideal $\maxid:=\symalg^+$.

\begin{notation} Throughout this section, we use the following notations.
    \begin{enumerate}
        \item The locus of points in $Z(F)$ with multiplicity $d-1$ is denoted by $\Delta_F$.
        \item For each $q\in \mathbb{Z}^{>0}$, we set $N_F^q:=\{g\in \mathbb{P}\symalg^+\mid g^q=0\}$. This is a closed subset of $\mathbb{P}\symalg^+$.
        \item For each $q\in \mathbb{Z}^{>0}$, we set $\widehat{N}_F^q:=\{g\in \symalg^+\mid g^q=0\}$. This is the affine cone over $N_F^q\subset \mathbb{P}\symalg^+$ with vertex $0\in \symalg^+$.
    \end{enumerate}
\end{notation}

\begin{proposition}\label{p_bound_for_Pm}
    For any positive integer $q$,
    \[\dim_\mathbb{C}\maxid\le\dim N_F^q+\dim_\mathbb{C}\maxid^q+1.\]
    \begin{proof}
        Let $f_1,...,f_r$ be a basis of $\maxid^q$. Consider the symmetric $q$-forms
        \[\mu_i:\underbrace{\maxid\times\cdots\times\maxid}_{q}\xrightarrow{\mu} \maxid^q\xrightarrow{f_i^*} \mathbb{C},\]
        where $\mu$ is the multiplication map and $f_i^*\in (\maxid^q)^*$ is the dual vector of $f_i$. Then the closed subset
        \[(\mu_1=\cdots=\mu_r=0)\subset \mathbb{P}\maxid\] 
        is equal to $N_F^q$. Hence
        \[\dim N_F^q\ge \dim\mathbb{P}\maxid-r=\dim_\mathbb{C}\maxid-1-r,\]
        as required.
    \end{proof}
\end{proposition}

\begin{proposition}\label{p_minimal_p} 
Let $p$ be the minimal integer such that $\maxid^{p+1}=0$. Then 
\begin{enumerate}
    \item If $p=1$, then $\dim_\mathbb{C}\maxid\le \dim N_F^2+1$.
    \item If $p\in \{2,3\}$, then $\dim_\mathbb{C}\maxid\le 2(\dim N_F^2+1)$.
\end{enumerate}
\begin{proof}
(1) If $p=1$ then $\mathbb{P}\maxid=N_F^2$. Thus $\dim_\mathbb{C}\maxid\le \dim N_F^2+1$.

(2) If $p\in \{2,3\}$ then $\mathbb{P}\maxid^2\subset N_F^2$, hence $\dim_\mathbb{C}\maxid^2\le \dim N_F^2+1$. By Proposition~\ref{p_bound_for_Pm}, $\dim_\mathbb{C}\maxid\le 2 (\dim N_F^2+1)$.
\end{proof}
\end{proposition}

We now estimate the minimal integer $p$ such that $\maxid^{p+1}\subset R_F$ vanishes.

\begin{lemma}\label{l_power_vanishes}
    Let $A$ be a commutative $\mathbb{C}$-algebra, $q\ge 1$ and $\idvar\subset A$ be an ideal such that $a^q=0$ for all $a\in \idvar$. Then $\idvar^q=0$.
    \begin{proof}
        It is enough to show that $a_1\cdots a_q=0$ for any $a_1,...,a_q\in \idvar$. Let 
        \[P=(a_1x_1+\cdots+a_qx_q)^q\in A[x_1,...,x_q].\] 
        By assumption, for any $(t_1,...,t_q)\in \mathbb{C}^q$,
            \[0=P(t_1,...,t_q)=\sum a_{i_1}\cdots a_{i_q}t_{i_1}\cdots t_{i_q}.\]
        Since $\mathbb{C}$ is infinite, $P$ is the zero polynomial. Since the coefficient of $x_1\cdots x_q$ in $P$ is $q! a_1\cdots a_q$, we conclude that $a_1\cdots a_q=0$.
    \end{proof}
\end{lemma}

\begin{proposition}\label{p_power_vanishes}
    Assume that $\Delta_F$ does not contain a linear subspace of $\mathbb{P}V$ of dimension $k$. Then $\maxid^{2k+1}=0$.
    \begin{proof}
        This is a direct consequence of Proposition \ref{p_power} and Lemma \ref{l_power_vanishes}.
    \end{proof}
\end{proposition}

\subsection{Structure of $N_F^2$} The locus $N_F^2$ encodes subtle geometric information about $\Delta_F$. Assuming that $\Delta_F$ has only finitely many projective lines, we make this relationship explicit by describing the irreducible components of $N_F^2$ in terms of lines in $\Delta_F$.

To describe the structure of $N_F^2$, we introduce the following definition.

\begin{definition}  
    A linear subspace $L\subset \mathbb{P}V$ is \textbf{generated by a symmetrizer of $F$}, if there is $g\in \symalg^+$ such that $g^2=0$ and $\mathbb{P}(\im g)=L$.
\end{definition}

\begin{theorem}\label{t_component}
    Assume that $\Delta_F$ has only finitely many lines.
    \begin{enumerate}
        \item Every irreducible component of $N_F^2$ is either a point or a line.
        \item If there is a line generated by a symmetrizer of $F$, then there is a one-to-one correspondence
        \[\{\text{(set-theoretic) irreducible components of $N_F^2$}\}\longleftrightarrow\{\text{lines generated by symmetrizers of $F$}\}.\]
    \end{enumerate}
\end{theorem}

\begin{lemma}\label{l_at_most_two_quasi-vertices}
    \begin{enumerate}
        \item If $Z(F)$ has three distinct quasi-vertices, then $\Delta_F$ contains a plane and therefore infinitely many lines.
        \item Let $g,h\in \symalg^+$ be linearly independent elements of rank $1$. Then a general linear combination of $g$ and $h$ has rank $2$.
    \end{enumerate}
    \begin{proof}
        (1) By Proposition \ref{p_quasi-vertices_conic}, a line in $Z(F)$ contains at most two quasi-vertices. Thus, if $Z(F)$ has more than two quasi-vertices, then by Proposition \ref{p_quasi-vertex_mult}, any three of the quasi-vertices generate a plane contained in $\Delta_F$.

        (2) Write $\im g=\langle gu\rangle$ and $\im h=\langle hv\rangle$. Suppose that $u$ and $v$ are proportional, say $u=v$. Since a linear combination of $g$ and $h$ is of rank at most $1$, for a general $(a:b)\in \mathbb{P}^1$,
        \[[gu], [hu],\text{ and }[(ag+bh)u]\]
        are three distinct quasi-vertices on the line $\mathbb{P}\langle gu,hu\rangle\subset Z(F)$. This contradicts Proposition~\ref{p_quasi-vertices_conic}. Hence $u,v$ are linearly independent. By Theorem \ref{t_quasi-vertices}(1), $\im g\ne \im h$. Hence the image of a general linear combination of $g$ and $h$ is $2$-dimensional.
    \end{proof}
\end{lemma}

We need the following to analyze the irreducible components of $N_F^2$.

\begin{proposition}\label{p_Gr}
    Assume that $\Delta_F$ has only finitely many lines. Consider the map $\Phi_2:\{g\in N_F^2\mid \rk g=2\}\to \operatorname{Gr}(2,V)$ given by $g\mapsto \im g$.
    Let $W\in \operatorname{Gr}(2,V)$.
    \begin{enumerate}
        \item Every element in $\Phi_2^{-1}(W)$ has the same kernel.
        \item We have
        \[\dim_\mathbb{C}\langle \Phi_2^{-1}(W)\rangle\le \begin{cases}
            1, \text{ if $\mathbb{P}W$ does not contain a quasi-vertex of $Z(F)$}\\
            2, \text{ otherwise}.
        \end{cases}\]
    \end{enumerate}
    \begin{proof}
        (1) Let $g,h\in \Phi_2^{-1}(W)$. By Proposition \ref{p_quasi-vertices_conic}, the projective line $\mathbb{P}W\subset Z(F)$ has at most two quasi-vertices. Thus a general point $[\xi]\in \mathbb{P}W$ satisfies $\operatorname{rank}_F(\xi)=2$. By Proposition~\ref{p_kernel}, $\ker g=\ker h$.

        (2) Fix an element $g\in \Phi_2^{-1}(W)$ and let $U=\ker g$. By (1), $U$ is contained in the kernel of every element of $\langle\Phi_2^{-1}(W)\rangle$. Hence every element in $\langle\Phi_2^{-1}(W)\rangle$ factors through $V/U$, and we obtain an embedding 
        \[\iota: \langle\Phi_2^{-1}(W)\rangle\hookrightarrow \operatorname{Hom}_\mathbb{C}(V/U,W)\cong \mathbb{C}^{2\times 2}.\]
        Let $X=\mathbb{P}(\iota\langle\Phi_2^{-1}(W)\rangle)\subset \mathbb{P}(\mathbb{C}^{2\times 2})$. The closed subset $Q\subset \mathbb{P}(\mathbb{C}^{2\times 2})\cong \mathbb{P}^3$ consisting of matrices with rank $1$ is a quadric surface cut out by the determinant. By Lemma \ref{l_at_most_two_quasi-vertices}(1), $Z(F)$ has at most two quasi-vertices. Thus $X\cap Q$ has at most two points, and so $\dim X\le 1$. Hence $\dim_\mathbb{C} \langle \Phi_2^{-1}(W)\rangle\le 2$. 
        
        Suppose that $\mathbb{P}W$ does not contain a quasi-vertex of $Z(F)$. If there is an element $g\in \langle \Phi_2^{-1}(W)\rangle$ of rank $1$, then since $\im g\subset W$, $\mathbb{P}W$ contains the $\mathbb{P}(\im g)$ quasi-vertex of $Z(F)$, which is absurd. Thus every element in $\langle\Phi_2^{-1}(W)\rangle$ has rank $2$. Therefore, $X\cap Q=\varnothing$ and $\dim X\le0$.
    \end{proof}
\end{proposition}

\begin{lemma}\label{l_B_W}
    For each $W\in \operatorname{Gr}(2,V)$, define a closed subset $B_W\subset \mathbb{P}\symalg^+$ by
    \[B_W:=\{g\in N_F^2\mid \im g\subset W\}.\]
    \begin{enumerate}
        \item Each of $B_W$ is a linear subspace of $\mathbb{P}\symalg^+$.
        \item Let $\Phi_2$ be the map defined in Proposition \ref{p_Gr}. If $\Phi_2^{-1}(W)\ne \varnothing$, then $B_W=\mathbb{P}(\langle \Phi_2^{-1}(W)\rangle)$.
    \end{enumerate}
    \begin{proof}
    (1) We first show that $gh=0$ for any $g,h\in B_W$. If $\rk g=1$ then by Proposition~\ref{p_quasi-vertex_stabilized}(2), $gh=0$ and hence every $\mathbb{C}$-linear combination of $g$ and $h$ is square-zero. Assume that $\rk g=\rk h=2$. Then $\im g=W=\im h$. Since $g^2=h^2=0$, we have $gh=0$. Since $\symalg$ is commutative, the affine cone
        \[\widehat{B}_W:=\{g\in \symalg^+\mid g^2=0\text{ and }\im g\subset W\}\subset \symalg^+\]
    over $B_W$ is a subspace of $\symalg^+$. Hence $B_W=\mathbb{P}(\widehat{B}_W)\subset \mathbb{P}\symalg^+$ is a linear subspace.

    (2) By assumption, $B_W$ has an element $\varphi\in B_W$ of rank $2$. Consider the affine cone $\widehat{B}_W\subset \symalg^+$ over $B_W$, which is a vector space by (1). We claim that $\widehat{B}_W=\langle\Phi_2^{-1}(W)\rangle$. If $g\in \widehat{B}_W$ is an element of rank $1$ then a general linear combination of $g$ and $\varphi$ has rank $2$, thus $\widehat{B}_W$ is generated by elements with rank $2$. Since $\im (f+h)\subset W$ for any $f,h\in \widehat{B}_W$, we see that $\widehat{B}_W$ is spanned by the elements of $\widehat{N}_F^2$ whose images are equal to $W$. Hence $\widehat{B}_W=\langle\Phi_2^{-1}(W)\rangle$. 
    \end{proof}
\end{lemma}

    \begin{proof}[Proof of Theorem \ref{t_component}] Assume that $\Delta_F$ has only finitely many lines of $\mathbb{P}V$.
    
         (1) Let $U\in \operatorname{Gr}(2,V)$ and suppose that $B_U$ does not have an element of rank $2$. By Lemma~\ref{l_at_most_two_quasi-vertices}(2), $B_U$ is either empty or consists of one element of rank $1$. We claim that either $N_F^2$ consists of at most one element, or
        \[N_F^2=\bigcup_{B_W \text{ has an element of rank $2$}}B_W.\]
         For this, it is enough to show that, 
         \begin{itemize}
            \item if $g\in N_F^2$ is of rank $1$ and if there is $B_W$ containing an element of rank $2$, then $g\in B_W$, and
            \item if $B_U, B_{U'}$ do not contain an element of rank $2$, then there is $B_W$ containing an element of rank $2$ such that $B_U,B_{U'}\subset B_W$.
        \end{itemize}
        For the first bullet let $\varphi\in B_W$ be an element of rank $2$. If $\im g\not\subset W=\im \varphi$, then since $\mathbb{P}(\im g)$ is a quasi-vertex and since $\mathbb{P}W\subset \Delta_F$, by Proposition~\ref{p_quasi-vertex_mult}, the plane $\mathbb{P}\langle W\oplus \im g\rangle$ belongs to $\Delta_F$, so $\Delta_F$ contains infinitely many lines. This is a contradiction, hence $g\in B_W$. To prove the second bullet, assume that there exist two distinct $B_U$ and $B_{U'}$, neither of which contains an element of rank $2$. Write $B_{U}=\{g\}$ and $B_{U'}=\{g'\}$. Since $g$ and $g'$ are distinct elements with rank $1$, by Lemma \ref{l_at_most_two_quasi-vertices}(2), we can choose a linear combination $h$ of $g$ and $g'$ with $\rk h=2$. Let $W=\im g\oplus \im g'=\im h$. Then $B_U, B_{U'}\subset B_W$. Moreover, since $\rk g=\rk g'=1$, by Proposition \ref{p_quasi-vertex_stabilized}(2), $h^2=0$. Thus, $h\in B_W$.
   
        Next, we show that every $B_W$ containing an element of rank $2$ is a component of $N_F^2$. By Lemma \ref{l_B_W}(1), every $B_W$ is irreducible. Since $\Delta_F$ has only finitely many lines, there are only finitely many $W\in \operatorname{Gr}(2,V)$ such that $B_W$ has an element of rank $2$ (since, if $B_W$ has an element of rank $2$ then $\mathbb{P}W\subset \Delta_F$). Thus the above union is finite. If $W_1\ne W_2$, and $\varphi_1\in B_{W_1}$, $\varphi_2\in B_{W_2}$ are elements of rank $2$, then $\im \varphi_1\ne \im \varphi_2$, so neither $B_{W_1}\subset B_{W_2}$ nor $B_{W_1}\supset B_{W_2}$. Hence we conclude that each of $B_W$ containing an element of rank $2$ is an irreducible component of $N_F^2$. That is, the above union is the (set-theoretic) decomposition of $N_F^2$ into its irreducible components. Combining Lemma \ref{l_B_W}(2) and Proposition \ref{p_Gr}(2) we see that each component $B_W$ of $N_F^2$ is either a projective line or a point.

        (2) Since there is a line generated by a symmetrizer of $F$, $N_F^2$ has an element of rank $2$. By the proof of (1), if $N_F^2$ has an element of rank $2$, then the irreducible components of $N_F^2$ are those $B_W$ containing an element of rank $2$. By definition of $B_W$, the map sending each component $B_W$ to the line $\mathbb{P}W$ generated by a symmetrizer is bijective.
    \end{proof}

We conclude this subsection with an example illustrating the geometry of $N_F^2$.

\begin{example}\label{e_reducible}
    Let $F\in \operatorname{Sym}^d_o(\mathbb{C}^5)^*$ be the symmetric form corresponding to $P = x_0x_2^2 + x_1^2x_2 + x_2x_3^2 + x_4^3 \in \mathbb{C}[x_0,\dots,x_4]$. Then 
    \[\symalg^+=\langle h\rangle\oplus \langle h^2\rangle\oplus \langle g\rangle,\]
    where
    \[h=\left(\begin{array}{ccc|cc}
        0&1&0&&\\
        &0&1&\\
        &&0&\\
        \hline
        &&&&\\
        &&&&
    \end{array}\right) \text{ and } g=\left(\begin{array}{ccc|cc}
        0&0&0&1&0\\
        &0&0&\\
        &&0&\\
        \hline
        &&1&&\\
        &&0&&
    \end{array}\right).\]
    The singular locus $\Delta_F$ of the cubic threefold $Z(F)\subset \mathbb{P}^4$ is the union of two lines meeting at the quasi-vertex $(1:0:0:0:0)$, and these lines are the images of the elements
    \[h+\sqrt{-1}g, \text{ and }h-\sqrt{-1}g\]
    which square to zero. Moreover, $N_F^2$ is the union of two lines in $\mathbb{P}\mathfrak{g}_F^+$:
    \[N_F^2=\mathbb{P}\left(\mathbb{C}h^2\oplus\mathbb{C}\cdot(h+\sqrt{-1}g)\right)\cup\mathbb{P}\left(\mathbb{C}h^2\oplus\mathbb{C}\cdot(h-\sqrt{-1}g)\right).\]
    These two components intersect at the unique point $h^2\in N_F^2$, which corresponds to the quasi-vertex.
\end{example}

\subsection{Proof of Theorem \ref{t_main} and Theorem \ref{t_main_2}}

\begin{proof}[Proof of Theorem \ref{t_main}] 
        Assume that $\Delta_F$ does not contain a line. By Theorem \ref{t_quasi-vertices}, $N_F^2$ consists of at most one point. Let $p$ be the minimal integer such that $\maxid^{p+1}=0$. By Proposition~\ref{p_power_vanishes}, $p\le 2$. If $p\le 1$ then by Proposition \ref{p_minimal_p}, $\dim_\mathbb{C}\maxid\le 1$ and so \[R_F\cong \mathbb{C}\text{ or }\mathbb{C}[x]/(x^2).\] Now if $p=2$ then the strict chain $\maxid\supsetneq \maxid^2\supsetneq \maxid^3=0$ and Proposition \ref{p_power_vanishes}, shows that $\dim_\mathbb{C}\maxid=2$ and that
        \[\dim_\mathbb{C}\maxid/\maxid^2=\dim_\mathbb{C}\maxid^2=1.\]
        Hence in this case,
        \[R_F\cong \mathbb{C}[x]/(x^3).\]
        This proves (1).
        
        By Proposition \ref{p_inequality}, $\dim G_F\le \dim_\mathbb{C}V-\max\{\operatorname{rank}g\mid g\in \mathfrak{g}_F^+\}+\dim_\mathbb{C}\mathfrak{g}_F^+\le \dim_\mathbb{C}V$. This proves (2). The following example shows that the bounds $\dim_\mathbb{C}\maxid\le 2$ and $\dim G_F\le \dim_\mathbb{C}V$ are sharp.
\end{proof}

\begin{example}\label{e_sharp}
    Consider the symmetric $d$-form $F$ corresponding to
    \[P=x_0x_2^{d-1}+(d-1)x_1^{2}x_2^{d-2}+\sum_{i=3}^nx_i^d\in \mathbb{C}[x_0,...,x_n]_d.\]
     As the partial derivatives are linearly independent, the hypersurface $Z(F)\subset \mathbb{P}^n$ is not a cone. The origin $(1:0:\cdots:0)\in Z(F)$ is the unique singularity of multiplicity $d-1$. Moreover, $\symalg$ is the vector space of all matrices of the form
    \begin{align*}
        \left(\begin{array}{ccc|ccc}
            u & s&t& \\
             &u&s \\
             &&u\\
             \hline
             &&&z_3&&\\
             &&&&\ddots&\\
             &&&&&z_n
        \end{array}\right).
    \end{align*} Thus
    \[\symalg^+=\mathbb{C}f\oplus \mathbb{C} f^2, \text{ where } f=\left(\begin{array}{ccc|cc}
        0&1&0&&\\
        &0&1&\\
        &&0&\\
        \hline
        &&&&\\
        &&&&
    \end{array}\right)\] and
    \[R_F=\mathbb{C}\oplus \mathfrak{g}_F^+\cong \mathbb{C}[x]/(x^3).\]
\end{example}

We now turn to the proof of Theorem \ref{t_main_2}. Suppose that $\Delta_F$ has only finitely many lines and that $\maxid^4\ne 0$, where $\maxid=\symalg^+$. In this situation, one can show that $\maxid^4$ is one-dimensional. Consequently, one may consider
\[Q:\maxid^2\times\maxid^2\to \maxid^4\cong \mathbb{C}\]
and use the fact that $\dim N_F^2\le 1$ to obtain the estimate $\dim_\mathbb{C}\maxid\le 5$. The following lemma shows that this estimate can be sharpened to the optimal bound $\dim_\mathbb{C}\maxid\le 4$.

\begin{lemma}\label{l_p=4}
    Assume that $\Delta_F$ has only finitely many lines, and let $\maxid=\symalg^+$. If $\maxid^4\ne 0$ then 
    \[\maxid=\bigoplus_{i=1}^4 \mathbb{C}f^i\]
    for some $f\in \maxid$.
    \begin{proof} 
        We divide the proof into $4$ steps.
    
        \emph{Step 1.} If $g^4=0$ for any $g\in \maxid$, then $\maxid^4=0$ by Lemma \ref{l_power_vanishes}. Thus there is $f\in \maxid$ such that $f^4\ne0$. Consider the chain
        \[\maxid\supsetneq \maxid^2\supsetneq \maxid^3\supsetneq \maxid^4\supsetneq 0.\]
        Since $\mathbb{C}f^3\oplus \mathbb{C}f^4\subset\maxid^3\subset \widehat{N}_F^2$, Theorem \ref{t_component} implies that $\dim \widehat{N}_F^2=2$. Thus
        \[\widehat{N}_F^2=\maxid^3=\mathbb{C}f^3\oplus \mathbb{C}f^4,\]
        and $\maxid^4=\mathbb{C}f^4$.
        \medskip
        
        \emph{Step 2.} We now show that $f$ has a unique nonzero Jordan block. By Proposition \ref{p_Jordan_blocks}, $f$ has at most two nonzero Jordan blocks. Since $f^4\ne0$ and $f^5=0$, the cyclic decomposition of $V$ with respect to $f$ is of the form
        \[V\cong \mathbb{C}[T]/(T^5)\oplus \mathbb{C}[T]/(T^r)\oplus \mathbb{C}^q\]
        where $1\le r\le 5$. Assume that $r> 1$. For $j=3,4$, let $u_j\in V$ be the element corresponding to $T^j\in \mathbb{C}[T]/(T^5)$ and $w\in V$ the element corresponding to $T^{r-1}\in \mathbb{C}[T]/(T^r)$. Then $u_3,u_4,w$ are linearly independent, $u_3\in \ker f^2$, $u_4\in \ker f$, $w\in \ker f$, and in particular,
        \[f^2u_3=f^2u_4=f^2w=0.\] 
        By Proposition \ref{p_multiplicity}, the plane $\mathbb{P}\langle u_3,u_4,w\rangle$ belongs to $\Delta_F$. Since we're assuming that $\Delta_F$ has only finitely many lines, this is a contradiction. Thus $r=1$ and $f$ has exactly one nonzero Jordan block. In particular, $\rk f=4$.
        \medskip

        \emph{Step 3.}  Let $h\in \symalg^+$. Choose a basis $\mathfrak{B}=\{\xi_0,...,\xi_n\}$ of $V$ such that the matrix representation of $f$ with respect to $\mathfrak{B}$ is
        \[[f]_\mathfrak{B}=\left(\begin{array}{c|c}
             J_5(0)\\
             \hline
             &
        \end{array}\right)\in \mathbb{C}^{(n+1)\times (n+1)},\]
        where $J_5(0)\in \mathbb{C}^{5\times 5}$ is the nilpotent matrix of Jordan form with minimal polynomial $x^5$. A direct computation using $fg=gf$ shows that
        \[[h]_\mathfrak{B}=\left(\begin{array}{c|c}
             \begin{array}{ccccc}
                  &  \\
                  &  \\
                  &p(J_5(0))  \\
                  &  \\
                  &
             \end{array}& \begin{array}{ccc}
                 &w^t&  \\
                 \hline
                  & \\
                  &  \\
                  &  \\
                  &
             \end{array} \\
             \hline
            \begin{array}{cccc|c}
                 &&&&  \\
                 &&&&v\\
                 &&&&
            \end{array} & A
        \end{array}\right)\]
        for some polynomial $p=a_1x+\cdots+a_4x^4\in \mathbb{C}[x]$, $v,w \in \mathbb{C}^{n-4}$ and $A\in \mathbb{C}^{(n-4)\times (n-4)}$.
        \medskip

        \emph{Step 4.} We show that $h\in\bigoplus_{i=1}^4\mathbb{C}f^i$. Choose $\epsilon\in \mathbb{C}$ such that $\epsilon +a_1\ne 0$. Then $(\epsilon f+h)^4\ne0$. Applying the previous two steps to $\epsilon f+h$, we see that $\rk(\epsilon f+h)=4$. If $A\ne 0$ then $\rk(\epsilon f+h)>4$. Thus $A=0$. Now choose any $a\in \mathbb{C}$ such that $a^2+w^tv=0$. Then for $\varphi_a:=af^2+(h-p(f))$, we have
        \[[\varphi_a]_\mathfrak{B}=\left(\begin{array}{c|c}
             \begin{array}{ccccc}
                  0&0&a&0&0  \\
                  &0&0&a&0  \\
                  &&0&0&a  \\
                  &&&0&0  \\
                  &&&&0
             \end{array}& \begin{array}{ccc}
                  &w^t& \\
                 \hline
                  & \\
                  &  \\
                  &  \\
                  &
             \end{array} \\
             \hline
            \begin{array}{cccccc|c}
                 &&&&&&  \\
                 &&&&&&v\\
                 &&&&&&
            \end{array} & 
        \end{array}\right),\]
        and $\varphi_a^2=0$. That is, $\varphi_a\in \widehat{N}_F^2=\mathbb{C}f^3\oplus \mathbb{C}f^4$. Thus $v=w=0$, and hence $h\in \bigoplus_{i=1}^4\mathbb{C}f^i$.
    \end{proof}
\end{lemma}
\medskip

    \begin{proof}[Proof of Theorem \ref{t_main_2}]
    Let $p$ be the minimal integer such that $\maxid^{p+1}=0$. By Theorem \ref{t_component}, $\dim N_F^2\le 1$.  By Proposition \ref{p_power_vanishes}, $p\le 4$. By Proposition \ref{p_minimal_p} and Lemma \ref{l_p=4}, we have
        \[\dim_\mathbb{C}\maxid \le \begin{cases}
            0, \text{ if $p=0$}\\
            2, \text{ if $p=1$}\\
            4, \text{ if $p\in \{2,3\}$}\\
            4, \text{ if $p=4$}.
        \end{cases}\] The sharpness of the bound $\dim_\mathbb{C}\maxid\le 4$ is established by Example \ref{e_sharp_2}.
        
        We now classify the possible isomorphism classes of $R_F$. Let $h_i=\dim_\mathbb{C}\maxid^i/\maxid^{i+1}$. Note that
        \begin{align}
            p=0&\implies R_F=\mathbb{C},\\
            p=1\text{ and }\dim_\mathbb{C}\maxid=1&\implies R_F\cong \mathbb{C}[x]/(x^2),\\
            p=1\text{ and }\dim_\mathbb{C}\maxid=2&\implies h_1=2\implies R_F\cong \mathbb{C}[x,y]/(x,y)^2.
        \end{align} Let $p\in \{2,3\}$. Since $\maxid^2\subset \widehat{N}_F^2$, and since $\dim\widehat{N}_F^2\le 2$ (by Theorem \ref{t_component}(1)), we see that $\dim_\mathbb{C}\maxid^2\le 2$.
        
        Suppose that $p=2$. Then $(h_1,h_2)$ is one of $(1,1), (2,1), (3,1) \text{ and } (2,2)$. A direct computation shows that
        \begin{align}
            (h_1,h_2)=(1,1)&\implies R_F\cong \mathbb{C}[x]/(x^3),\\
            (h_1,h_2)=(2,1)&\implies R_F\cong \mathbb{C}[x,y]/((x,y)^3,y^2,xy)\text{ or }\\& \quad\quad\quad\quad\mathbb{C}[x,y]/((x,y)^3,x^2,y^2).
        \end{align}
        Assume $(h_1,h_2)=(2,2)$. By a direct calculation, $R_F$ is isomorphic to either $\mathbb{C}[x,y]/((x,y)^3,y^2)$ or $\mathbb{C}[x,y]/((x,y)^3,xy)$. However, in the first ring $\mathbb{C}[x,y]/(y^2,(x,y)^3)$, the square of every element in the subspace $\langle y,x^2,xy\rangle$ vanishes. Since $\dim \widehat{N}_F^2\le 2$, $R_F$ cannot be isomorphic to this ring. Thus
        \begin{align}(h_1,h_2)=(2,2)\implies R_F\cong \mathbb{C}[x,y]/((x,y)^3,xy)=\mathbb{C}[x,y]/(x^3,y^3,xy).\end{align}
        If $(h_1,h_2)=(3,1)$, then $\dim_\mathbb{C}\maxid=4$ and $\dim_\mathbb{C}\maxid^2=1$, so $N_F^2$ is a quadratic surface in $\mathbb{P}\maxid\cong \mathbb{P}^3$. As $\dim N_F^2\le 1$, this case cannot happen. 
        
        Let $p=3$. Then the conditions $\dim_\mathbb{C}\maxid\le 4$ and $\dim_\mathbb{C}\maxid^2\le 2$ force $(h_1,h_2,h_3)$ to be one of $(1,1,1)$ and $(2,1,1)$. Clearly, 
        \begin{align}(h_1,h_2,h_3)=(1,1,1)\implies R_F\cong \mathbb{C}[x]/(x^4).\end{align} If $(h_1,h_2,h_3)=(2,1,1)$, then we can choose $x,y\in \maxid$ so that
        \[\maxid/\maxid^2=\langle x,y\rangle,\quad \maxid^2/\maxid^3=\langle x^2\rangle\quad \text{and}\quad \maxid^3=\langle x^3\rangle.\]
        Say $xy=\lambda x^2+\mu x^3\in \maxid^2$, where $\lambda,\mu\in \mathbb{C}$. Changing the coordinate $y\mapsto y+\lambda x+\mu x^2$, we may assume that $xy=0$. Write $y^2=ax^2+bx^3\in \maxid^2$ for some $a,b\in \mathbb{C}$. Multiplying both sides by $x$, we have $a=0$, so $y^2=bx^3$. This shows that 
        \[R_F\cong \mathbb{C}[x,y]/(xy,y^2-bx^3, (x,y)^4)=\mathbb{C}[x,y]/(xy,y^2-bx^3,x^4).\]
        If $b=0$, then the square of every element in the $3$-dimensional subspace
        \[\langle x^2,x^3,y\rangle\subset \maxid\]
        vanishes, hence $\dim N_F^2\ge 2$. Since $\dim N_F^2\le 1$, we have $b\ne 0$. Thus,
        \begin{align}(h_1,h_2,h_3)=(2,1,1)\implies R_F\cong \mathbb{C}[x,y]/(xy,y^2-x^3,x^4).\end{align}
        Finally, Lemma \ref{l_p=4} shows that \begin{align}p=4\implies R_F\cong \mathbb{C}[x]/(x^5).\end{align}
        This proves (1). The sharpness of the bound $\dim_\mathbb{C}\maxid\le 4$ is established in Example \ref{e_sharp_2}.
        
        For (2), let $b=\max\{\rk g\mid g\in \symalg^+\}$. By Lemma~\ref{l_power_vanishes}, $\maxid$ has an element such that $f^p\ne 0$. Since $\rk f\ge p$, we see that $b\ge p$. By Proposition \ref{p_inequality},
        \begin{align*}
            \dim G_F\le\dim_\mathbb{C}V-b+\dim_\mathbb{C}\maxid\le \dim_\mathbb{C}V-p+\dim_\mathbb{C}\maxid.
        \end{align*}
        By the classification of $R_F$ in (1), we have
        \begin{align*}
            \dim_\mathbb{C}\maxid-p\begin{cases}
                =2, &\text{if }R_F\cong \mathbb{C}[x,y]/(xy,x^3,y^3) \\
                \le 1,&\text{otherwise}.
            \end{cases}
        \end{align*} Thus if $R_F\not\cong \mathbb{C}[x,y]/(xy,x^3,y^3),$ then $\dim G_F\le \dim_\mathbb{C}V+1$. 
        
        Assume that $R_F\cong \mathbb{C}[x,y]/(xy,x^3,y^3)$. We claim that $b\ge 3$. There is nothing to prove if one of the elements $x,y\in \maxid$ is a symmetrizer of rank $\ge 3$. Suppose that $\rk x,\rk y\le 2$. Then $x^2,y^2\in \maxid^2$ are linearly independent symmetrizers of rank $1$, so they correspond to two distinct quasi-vertices. By Lemma \ref{l_at_most_two_quasi-vertices} and Theorem \ref{t_quasi-vertices}, up to scalar multiples, $x^2,y^2$ are all the elements in $\maxid$ with rank $1$. Thus the nilpotent symmetrizer $(x+y)^2$ has rank at least $2$ and \[b\ge\rk(x+y)\ge 3.\] Since $\dim_\mathbb{C}\maxid=4$, $\dim G_F\le \dim_\mathbb{C}V-3+4=\dim_\mathbb{C}V+1$.
    \end{proof}

\begin{example}\label{e_sharp_2}
    Let $F\in \operatorname{Sym}_o^dV^*$ corresponding to the homogeneous polynomial
    \begin{align*}
        P=x_0x_4^{d-1}+(d-1)x_1x_3x_4^{d-2}+\frac{d-1}{2}x_2^2x_4^{d-2}+\frac{(d-1)(d-2)}{2}x_2x_3^2x_4^{d-3}\\
        +\frac{(d-1)(d-2)(d-3)}{24}x_3^4x_4^{d-4}+(x_5^d+\cdots+x_n^d).
    \end{align*}
    Then $\Delta_F=(x_2=\cdots=x_n=0)\cong \mathbb{P}^1$ and
    \[\symalg^+=\bigoplus_{i=1}^4\mathbb{C}f^i, \text{ where } f=\left(\begin{array}{ccccc|ccc}
            0&1&&&&&&\\
             &0&1&&&\\
             &&0&1&&\\
             &&&0&1&\\
             &&&&0&\\
             \hline
             &&&&&\\
             &&&&&&\\
        \end{array}\right).\] Moreover,
        \[R_F=\mathbb{C}\oplus\mathfrak{g}_F^+\cong \mathbb{C}[x]/(x^5).\]
\end{example}

\begin{remark}
    A direct computation of $N_F^2$ gives the following refinement of Theorem~\ref{t_component}(1). If $\Delta_F$ has only finitely many lines, then for each cases of $R_F$, we have
    \begin{center}\begin{tabular}{c|c}
        $R_F$ & $N_F^2$ \\
         \hline
        $\mathbb{C}$ & $\varnothing$\\
        $\mathbb{C}[x]/(x^2)$& a point \\
        $\mathbb{C}[x]/(x^3)$& a point\\
        $\mathbb{C}[x,y]/(x,y)^2$& $\mathbb{P}^1$ \\
        $\mathbb{C}[x]/(x^4)$& $\mathbb{P}^1$\\
        $\mathbb{C}[x,y]/(xy,x^3,y^2)$& $\mathbb{P}^1$\\
        $\mathbb{C}[x,y]/(x^2,y^2,(x,y)^3)$& two $\mathbb{P}^1$'s meeting at a point\\
        $\mathbb{C}[x]/(x^5)$& $\mathbb{P}^1$\\
        $\mathbb{C}[x,y]/(xy,x^3,y^3)$& $\mathbb{P}^1$\\
        $\mathbb{C}[x,y]/(xy,y^2-x^3,x^4)$& $\mathbb{P}^1$
    \end{tabular}.
    \end{center}
\end{remark}

\section{\bf Corollaries}\label{section_6}

\begin{corollary}\label{c_st}
    Let $P\in \mathbb{C}[x_0,...,x_n]_d$ and $H_P$ its Hessian matrix. The following are equivalent.
    \begin{enumerate}
        \item The polynomial $P$ is of Sebastiani--Thom type (i.e., after a linear change of coordinates, $P$ is written as $P_1+P_2$ for some $P_1\in \mathbb{C}[x_0,...,x_r]_d$ and $P_2\in \mathbb{C}[x_{r+1},...,x_n]_d$ with $0\le r<n$).
        \item There is a matrix $A\in \mathbb{C}^{(n+1)\times (n+1)}$ with at least two distinct eigenvalues, such that $H_P\cdot A$ is symmetric.
    \end{enumerate}
    \begin{proof}
        Let $F\in \operatorname{Sym}^dV^*$ be the symmetric $d$-form corresponding to $P$. By \cite[Theorem 1.6 (i)]{Hwang}, $P$ is of Sebastiani--Thom type if and only if $\symalg^\times\ne0$ if and only if $\symalg$ has an element with at least two distinct eigenvalues. By Proposition \ref{p_equivalent} and Remark~\ref{r_Hessian_matrix}, this is equivalent to the existence of a matrix $A \in \mathbb{C}^{(n+1)\times (n+1)}$ with at least two eigenvalues, such that $H_P\cdot A$ is symmetric.
    \end{proof}
\end{corollary}

\begin{corollary}
    Let $d\ge4$ and $F\in \operatorname{Sym}^d_oV^*$. Assume that $Z(F)$ has only isolated singularities. Then $\dim_\mathbb{C}\symalg^+\le 1$. The equality holds if and only if $Z(F)$ has a quasi-vertex as its unique singular point.
    \begin{proof}
        By Theorem \ref{t_main}, $\dim_\mathbb{C}\symalg^+=\max\{\rk g\mid g\in\symalg^+\}\le 2$. But $\symalg^+$ cannot have an element of rank $2$, since if $f$ is such an element, then $f^3=0$ and by Proposition \ref{p_multiplicity}, every point on the line $\mathbb{P}(\operatorname{im}f)$ is a point of $Z(F)$ with multiplicity $d-2\ge 2$. Thus $\dim_\mathbb{C}\symalg^+\le 1$. The second assertion follows from Theorem \ref{t_quasi-vertices} (2).
    \end{proof}
\end{corollary}

In contrast to the case $d\ge 4$, a cubic hypersurface with an isolated singularity can have $\dim_{\mathbb C}\symalg^+=2$.

\begin{corollary}
    Let $Z(F)\subset \mathbb{P} ^3$ be a cubic surface with isolated singularities. Assume that $Z(F)$ is not a cone.
    \begin{enumerate}
        \item The hypersurface $Z(F)$ has a singularity of type $D_4$ or $D_5$ if and only if $\dim_\mathbb{C}\symalg^+=1$.
        \item The hypersurface $Z(F)$ has a singularity of type $E_6$ if and only if $\dim_\mathbb{C}\symalg^+=2$.
    \end{enumerate}
        \begin{proof}
         For cubic hypersurfaces, the corank of a singular point $[\xi]$ is $n-\operatorname{rank}_F(\xi)$. By \cite[Lemma 4]{BW} and the table of \cite[page 255]{BW} shows that, if $p$ is a singularity on a cubic surface with isolated singularities that is not a cone, then
         \begin{itemize}
             \item the corank of $p$ is $2$ if and only if $p$ is of type $D_4,D_5$ or $E_6$,
             \item $p$ is of $E_6$ type if and only if there is a unique line on the surface passing through $p$.
         \end{itemize}
         Moreover, by the proof of \cite[Lemma 4]{BW}, a cubic surface having an $E_6$ singularity is projectively equivalent to the surface defined by $P=x_0x_2^2+x_1^2x_2+x_3^3$. Thus, if $Z(F)$ has an $E_6$ singularity then by Example \ref{e_sharp}, $\dim_\mathbb{C}\symalg^+=2$. Thus it is enough to prove that if $\dim_\mathbb{C}\symalg^+=2$ then there is a unique line passing through the singularity of $Z(F)$. This follows from the following lemma.
        \end{proof}

    \begin{lemma}\label{l_square_line}
     Let $Z(F)\subset \mathbb{P}^3$ be a surface of degree $d$ that is not a cone. If $[\xi]=\mathbb{P}(\im f^2)$ for some $f\in \symalg^+$ then there is a unique line in $Z(F)$ passing through $[\xi]$.
    \begin{proof}
        By Step 2 of the proof of Theorem \ref{t_main}, there is a basis $\mathfrak{B}=\{\xi_0,\xi_1,\xi_2,\xi_3\}$ of $V$ such that the matrix representation of $f$ with respect to $\mathfrak{B}$ is of the Jordan form
        \[\left(\begin{array}{ccc|c}
            0 &1&0&  \\
             & 0&1&\\
             &&0&\\
             \hline
             &&&0
        \end{array}\right).\]
        By Theorem \ref{t_quasi-vertices}, $[\xi_0]$ is a quasi-vertex, and so $\ker h_F(\xi_0)$ is $3$-dimensional. Let $H$ be the hyperplane section $\mathbb{P}(\ker h_F(\xi_0))\cap Z(F)$. By the proof of Proposition \ref{p_quasi-vertex_lines through x}, the underlying set of $H$ is the locus of lines on $Z(F)$ passing through $[\xi_0]$. So it is enough to prove that $H=dL$ for some line $L$. We claim that this holds for the line $L:=\mathbb{P}(\im f)=\mathbb{P}\langle \xi_0,\xi_1\rangle$.
        
        Note that for $i=0,1,3$,
        \[h_F(\xi_0)(\xi_i)=h_F(f^2\xi_2)(\xi_i)=h_F(x_2)(f^2\xi_i)=0.\]
        As $\operatorname{rank}_F(\xi)=1$, $\ker h_F(\xi_0)=\langle \xi_0,\xi_1,\xi_3\rangle$. Let $G=F|_{\ker h_F(\xi_0)}$ be the restriction. Then $H=Z(G)\subset \mathbb{P}(\ker h_F(\xi_0))$. Let $w_1,...,w_{d-1}\in \ker h_F(\xi_0)$. As each $w_i$ is a linear combination of $\xi_0,\xi_1,\xi_3$, for any $\alpha_0,\alpha_1\in \mathbb{C}$, 
        \[G(\alpha_0\xi_0+\alpha_1\xi_1,w_1,...,w_{d-1})=F(\alpha_0\xi_0+\alpha_1\xi_1,w_1,...,w_{d-1})\] 
        is a linear combination of the terms of the form
        \[F(\xi_i,\xi_j,\xi_k,...), \text{ where } i,j\in \{0,1\} \text{ and } k\in \{0,1,3\}.\]
        Note that $f^3=0$, $\xi_0,\xi_1\in \operatorname{im}f$ and $\xi_3\in \ker f$. Hence all of these terms vanish. That is, for every closed point $z\in L\subset H$, we have $\operatorname{mult}_z(G)=d$. Since $\deg H=d$, we conclude that $H=dL$.
    \end{proof}
\end{lemma}
\end{corollary}

\begin{corollary}
    Let $Z(F) \subset \mathbb{P}^4$ be a cubic threefold with isolated singularities, and assume that $Z(F)$ is not a cone. Then $\symalg^+\ne 0$ if and only if $Z(F)$ contains a singularity of one of the following types: $U_{12}$, $S_{11}$, $T_{444}$, $Q_{10}$, $T_{344}$, $T_{334}$ or $T_{333}$ (for normal forms of these singularities, see p.246 of \cite{AGV12})
    \begin{proof}
        Note that $[\xi]\in Z(F)$ is a quasi-vertex if and only if the corank of $[\xi]$ is $4-\operatorname{rank}_F(\xi)=4-1=3$. By \cite[Appendix D, Table 7]{Viktorova}, $[\xi]\in Z(F)$ has corank $3$ if and only if it is one of the types above. We are done by Theorem \ref{t_quasi-vertices}(2).
    \end{proof}
\end{corollary}

\end{document}